\documentclass[11pt,reqno]{amsart}
\usepackage[utf8]{inputenc}
\usepackage{MnSymbol}
\usepackage{fullpage}
\usepackage{color}
\usepackage{enumitem}
\usepackage{tikz}
\usepackage{capt-of}
\usepackage{tikzscale}

\newtheorem{teo}{Theorem}[section]

\newtheorem{coro}[teo]{Corollary}
\newtheorem{lema}[teo]{Lemma}

\newtheorem{prop}[teo]{Proposition}
\newtheorem{defi}[teo]{Definition}
\newtheorem{ex}[teo]{Example}
\newtheorem{question}[teo]{Question}

\DeclareMathOperator{\N}{\mathbb{N}}
\DeclareMathOperator{\Nt}{[\mathbb{N}]^2}

\DeclareMathOperator{\Q}{\mathbb{Q}}

\DeclareMathOperator{\h}{hom}
\DeclareMathOperator{\ch}{\overline{hom}}

\newcommand{\UR}{\mathcal{R}}
\newcommand{\cantor}{2^{\N}}

\definecolor{mygray}{gray}{0.8}

\title{Reconstruction of a coloring from its homogeneous sets}

\author{C. Pi\~{n}a}
\address{Departamento de Matem\'aticas, Facultad de Ciencias, Universidad de los Andes, C.P. 5101,  M\'erida- Venezuela.\newline
\indent Escuela de Matem\'aticas, Universidad Industrial de Santander, C.P. 680001, Bucaramanga - Colombia.}
\email{cpinarangel@gmail.com}

\author{C. Uzc\'{a}tegui}
\address{Escuela de Matem\'aticas, Universidad Industrial de Santander, C.P. 680001, Bucaramanga - Colombia.}
\email{cuzcatea@saber.uis.edu.co}

\thanks{The first author was supported by the Postdoctoral Program of Vicerrector\'{i}a de Investigaci\'{o}n y Extensi\'{o}n of the Universidad Industrial de Santander.}

\date{\today}

\begin{document}

\begin{abstract}
We study  the following reconstruction problem for colorings.  Given a  countable set $X$ (finite or  infinite), a coloring on $X$ is a function $\varphi: [X]^{2}\to \{0,1\}$, where $[X]^{2}$ is the collection of all  2-elements subsets of  $X$. A set $H\subseteq X$ is homogeneous for $\varphi$ when $\varphi$ is constant on $[H]^2$. Let $\h(\varphi)$ be the collection of all homogeneous sets for $\varphi$. The coloring $1-\varphi$ is called the complement of $\varphi$. We say that $\varphi$ is {\em reconstructible} up to complementation  from its homogeneous sets, if for any coloring $\psi$ on $X$ such that $\h(\varphi)=\h(\psi)$ we have that either $\psi=\varphi$ or $\psi=1-\varphi$.    We present several conditions for reconstructibility and non reconstructibility. For $X$ an infinite countable set, we show that there is a Borel way to recovering  a coloring from its homogeneous sets.

\end{abstract}


\maketitle

\section{Introduction}

In this paper we study the following reconstruction problem for colorings. Given a countable set $X$ (finite or infinite), a coloring on $X$ is a function $\varphi: [X]^{2}\to \{0,1\}$, where $[X]^{2}$ is the collection of 2-element subsets of  $X$. Let $\h(\varphi)$ be the homogeneous sets for $\varphi$; that is, the collection of $H\subseteq X$ such that $\varphi$ is constant on $[H]^2$. Clearly,  $\h(\varphi)=\h(1-\varphi)$. We say that $\varphi$ is {\em reconstructible} up to complementation from its homogeneous sets,  if for any coloring $\psi$ on $X$ such that $\h(\varphi)=\h(\psi)$ we have that either $\psi=\varphi$ or $\psi=1-\varphi$.  
In the terminology  of graphs, we are  talking about  graphs that can be reconstructed (up to complementation) from the collection of their  cliques and independent sets.

This type of reconstruction  problem  was considered long time ago in \cite{Cameron1993} for finite graphs but apparently  was not pursued  any further. A somewhat similar problem was addressed in \cite{Pouzetetal2013,Kaddour2019,Pouzetetal2011}. They analyzed a variant of the well known  graph reconstruction conjecture (see \cite{bondy1991}), and studied conditions under which a pair of  graphs with the same homogeneous sets are isomorphic up to complementation. In this paper we study conditions under which a pair of graphs with the same homogeneous sets are equal up to complementation.

An example of a reconstructible coloring is given by the random graph. We extract from this example a general method for showing reconstrutibility which is quite useful. Suppose that for every $F\subseteq X$ with $|F|=4$ there is $Y\supseteq F$  such that  the restriction of $\varphi$ to $[Y]^2$ is reconstructible, then $\varphi$ is reconstructible.   In particular, whenever a  coloring $\varphi$ on $\N$ has infinitely many initial segments which are reconstructible, then $\varphi$ itself is reconstructible.  

The first example that we found of a non-reconstructible coloring is given by a partition of  $\N$ into two infinite sets. We associate to this partition  a coloring $\varphi$ where $\varphi(\{x,y\})=1$ if and only if  both $x$ and $y$ belong to the same part of the partition.   This example satisfies a very simple criterion for non reconstructibility: If there is a pair $\{x,y\}$ (an edge) such that $\varphi(\{x,z\})=1-\varphi(\{y,z\})$ for all $z\nin \{x,y\}$, then $\varphi$  is non-reconstructible. The reciprocal is not true. 
Such pairs (edges) will be called {\em critical}. We  show a characterization of  colorings that admits  a critical pair. 

In the example mentioned above of a coloring associated to a partition of $\N$ into two parts, the collection of its homogeneous sets has exactly two maximal elements with respect to inclusion. Motivated by that, we present some results relating  the structure of the family of maximal homogeneous sets to the reconstruction problem. 

In the last section of the paper we study the reconstruction problem  from a  descriptive set theoretic point of view.   For instance,  the collection of reconstructible colorings on $\N$ is a dense $G_\delta$ subset of the space of colorings $\{0,1\}^{[\N]^2}$, that is, from the Baire category point of view, almost every coloring is reconstructible. We can regard $\h(\varphi)$ as a closed subset of the Cantor space $\{0,1\}^{\N}$ (which will be denoted, as usual, by $\cantor$), thus  as an element of the hyperspace $K(\cantor)$, which is a Polish space endowed with the usual Vietoris topology. We show that there is a Borel way to recover a coloring from its homogeneous sets. More precisely,  there is Borel map $f:K(\cantor)\to \{0,1\}^{[\N]^2}$ such that $f(\h(\varphi))$ is a reconstruction of $\varphi$, i.e., 
$\h(f(\h(\varphi))=\h(\varphi)$.

To finish this introduction we comment about our original motivation. A collection $\mathcal{H}$ of subsets of $\N$ is  {\em tall}, if for every infinite set $A\subseteq \N$, there is an infinite set $B\in \mathcal{H}$ such that $B\subseteq A$.  Ramsey's Theorem says that $\h(\varphi)$ is tall for every coloring $\varphi$ on $\N$. 
Some  tall families admit a Borel selector, that is,  a Borel map such that given an infinite set $A$, the map selects an  infinite subset of $A$ belonging to the tall family (\cite{GrebikUzca2018}). 
The collection $\h(\varphi)$  is an important example of  a tall family admitting a Borel selector (\cite{GrebikUzca2018,HMTU2017}).  It is an open problem to find a  characterization of  those tall Borel  families that admit a Borel selector. A quite related question is to characterize when a tall Borel family $\mathcal{H}$ admits a coloring $\varphi$ such that $\h(\varphi)\subseteq \mathcal{H}$. In other words, when is it possible to extract from such  tall family $\mathcal{H}$ a  coloring $\varphi$ such that $\h(\varphi)\subseteq \mathcal{H}$? These considerations  lead naturally to a Borel reconstruction problem: Suppose $\mathcal{H}=\h(\varphi)$, can we recover  from $\mathcal{H}$, in a Borel way,  a coloring $\psi$  such that   $\h(\psi)=\mathcal{H}$?  In the last section of the paper we show that the answer is positive.

\section{Preliminaries}

We will use standard notation from set theory. Throughout the article, $X$ will denote a countable (finite or infinite) set. Given $k\in\N$, we will denote by $[X]^{k}$ the collection of all subsets of $X$ of size $k$, by $[X]^{<k}$ the subsets of $X$ of size strictly less than $k$, and by $[X]^{\leq k}$ the union $[X]^{k}\cup[X]^{<k}$. The collection of all finite subsets of $X$ will be denoted by $[X]^{<\omega}$.  $X^{<\omega}$  denotes the collection of all finite sequences of elements of $X$  and $X^{\leq n}$ the collection of  sequences of length at most $n$ of elements of $X$.

A {\em coloring} on a set $X$, is any mapping $\varphi:[X]^2\to \{0,1\}$. Whenever is clear from the context, we  identify 2 with $\{0,1\}$. For instance,  the collection of all colorings $\{0,1\}^{[X]^2}$  will be denoted by $2^{[X]^2}$. Given a coloring $\varphi$ on $X$ and $Y\subseteq X$, we will denote by $\varphi|_Y$ the restriction of $\varphi$ to $[Y]^2$. We say that $\psi$ extends $\varphi$, and write $\varphi\subseteq \psi$, whenever $\varphi$ is a coloring on $Y$, $\psi$ is a coloring on $X$, $Y\subseteq X$ and $\psi|_Y=\varphi$. We write $\varphi\subset \psi$ when $\varphi\subseteq \psi$ and $\varphi\neq \psi$. 
For $X$ infinite, the family of colorings  $2^{[X]^2}$ will be seen as a topological space with the usual product topology which makes it homeomorphic to $2^X$.

A partition of $X$ is a collection $(A_i)_{i\in I}$ of non empty subsets of $X$ such that $I\subseteq \N$, $X=\bigcup_{i\in I}A_i$, and $A_i\cap A_j=\emptyset$ for every $i\neq j$ in $I$. Given $X=\bigcup_{i\in I}A_i$ a partition of $X$, we let {\em the coloring associated to the partition} be the mapping $\varphi:[X]^2\to 2$ defined by $\varphi(\{x,y\})=1$ if and only if $x,y\in A_i$ for some $i\in I$. 
Given a linear ordering $(X,<)$ and $\{n,m\}\in [X]^2$, we denote by $\{n,m\}_<$ the fact that $n<m$. If $e=\{r_n\}_n$ is an enumeration of $\Q$, the Sierpi\'{n}ski coloring $\varphi_e:[\N]^2\to 2$, associated to $e$, is defined by $\varphi_e(\{n,m\}_<)=1$ if and only if $r_n<r_m$. 

The  {\em random graph} $R=\langle \N, E\rangle$ (see \cite{Cameron1997}) has the following extension property. Given two finite disjoint subsets $A, B$  of $\N$, there is $n\in \N$ such that $\{x,n\}\in E$ for all $x\in A$ and $\{y,n\}\not\in E$ for all $y\in B$. 
This makes $R$  universal in the following sense.  Given a graph
$\langle{\N,G}\rangle$, there is a subset $X\subseteq\N$ such
that $\langle{\N,G}\rangle$ and $\langle {X, E|_X}\rangle$ are isomorphic.  

Given a coloring $\varphi:[X]^2\to 2$, we say that $H\subseteq X$ is {\em $i$-homogeneous (for $\varphi$)} if $\varphi([H]^2)=\{ i\}$ for  $i\in\{0,1\}$. This notion is clearly trivial if $|H|=2$, so we assume that an homogeneous set has at least 3 elements. Denote by $\h(\varphi)$ the set of homogeneous sets for $\varphi$; that is, 
$$
\h(\varphi)=\{H\subseteq X:\varphi\text{ is constant on }[H]^2\}.
$$

\begin{prop}
\label{triangulos}
Let $\varphi$ and $\psi$ be two colorings on a set $X$. Then  $\h(\varphi)=\h(\psi)$ if and only if  $\h(\varphi)\cap [\N]^{3}=\h(\psi)\cap [\N]^{3}$.
\end{prop}

\begin{proof}
Suppose $\h(\varphi)\cap [\N]^{3}=\h(\psi)\cap [\N]^{3}$ and let $H$ be a homogeneous set for  $\varphi$. Let $\{x,y\},\{w,z\}$ be two different pairs in $[H]^2$, by hypothesis $\{x,y,w\}$ and $\{y,w,z\}$ are $\psi$-homogeneous, hence $H$ is $\psi$-homogeneous. 
\end{proof}

It is clear that if $\varphi$ is the coloring associated to the partition $X=\bigcup_{i\in I}A_i$, then $\h(\varphi)=\{H:H\subseteq A_i\text{ for some }i\in I\}\cup\{H:|H\cap A_i|\leq 1\text{ for every }i\in I\}$. On the other hand, if $\varphi_e$ is the Sierpi\'{n}ski coloring associated to  an enumeration $e=\{r_n\}_n$ of $\Q$, then $H\in\h(\varphi_e)$ if and only if $H$ is monotone respect to $e$; that is, if either $r_n<r_m$ for every $n<m$ in $H$, or $r_n\geq r_m$ for every $n<m$ in $H$. In general, it is well known,  that if $|X|\geq 6$, there is an homogeneous set of size 3. Furthermore, we recall that  Ramsey's Theorem states that every coloring $\varphi:[X]^2\to 2$ on an infinite set $X$ has an infinite homogeneous set.

A coloring $\varphi:[X]^2\to 2$ is said to be  {\em reconstructible} (up to complementation) from its homogeneous sets if given a coloring $\psi:[X]^2\to 2$ such that $\h(\varphi)=\h(\psi)$, we have that either $\varphi=\psi$ or $\varphi=1-\psi$.  Let $\UR$ be the collection of all reconstructible colorings, and let $\neg \UR$ be its complement. We will call a coloring {\em non-reconstructible} if it belongs to $\neg\UR$. Since $\h(\varphi)=\h(1-\varphi)$, we have that $\varphi\in\UR$ if and only if $1-\varphi\in\UR$. Finally, given $\varphi,\psi\in 2^{[X]^2}$, we say that $\psi$ is a {\em reconstruction} of $\varphi$, if $\h(\psi)=\h(\varphi)$ and we say it is a non-trivial reconstruction if in addition  $\psi\neq\varphi$ and $\psi\neq 1-\varphi$.

\section{Reconstructible colorings}

The aim of this section is  to present some sufficient conditions for the reconstructibility of a coloring.  On the one hand, we shall see that in order to determine if a coloring belongs to $\UR$, it is enough to ensure that some finite restrictions do. On the other, we will introduce properties $E_0$ and $E_1$, and we will see that any coloring with any of these properties is in $\UR$.

\subsection{Finitistic conditions for reconstructibility}\hfill

\medskip

Our first result is a very useful  criterion for reconstructibility. 

\begin{prop}
\label{4suffices}
Let $\varphi$ be a coloring on $X$. If for every $F\in [X]^{\leq 4}$ there is $Y\subseteq X$ such that $F\subseteq Y$ and $\varphi|_Y\in \UR$, then $\varphi\in \UR$.
\end{prop}

\begin{proof} Let $\psi$ be a coloring on $X$ such that $\h(\varphi)=\h(\psi)$. Suppose that for every $F\in [X]^{\leq 4}$ there is $Y\subseteq X$ such that $F\subseteq Y$ and $\varphi|_Y\in \UR$; and that there are $x,y\in X$ such that $\varphi(\{x,y\})=\psi(\{x,y\})$. 
We will show that $\varphi=\psi$.  Let $w,z\in X$ with $\{x,y\}\neq\{z,w\}$. By hypothesis, there is $Y\subseteq X$ such that $\{x,y,w,z\}\subseteq Y$ and $\varphi|_Y\in \UR$. We have $\h(\varphi|_Y)=\h(\psi|_Y)$, $\varphi|_Y\in \UR$   and $\varphi(\{x,y\})=\psi(\{x,y\})$, therefore $\varphi|_Y=\psi|_Y$. In particular, $\varphi(\{w,z\})=\psi(\{w,z\})$ and we are done.
\end{proof}

There are  colorings $\varphi\in \UR$ such that $\varphi|_F\not\in\UR$ for some $|F|\leq 4$ (see Example \ref{4sufficesB}).

\begin{coro}
Let $\varphi$ be a coloring on $\N$. Suppose that for infinitely many $n$, $\varphi |_{\{0, \cdots, n\}}\in \UR$, then $\varphi\in \UR$. 
\end{coro}

The previous result naturally suggests the following problem. 

\begin{question}
\label{segmentos-UR} Let $\varphi$ be a reconstructible coloring on $\N$ and $F\subseteq\N$ be  a finite set. Is there  a finite set $G\supseteq F$ such that $\varphi |_G\in \UR$?
\end{question}

\medskip

Proposition \ref{4suffices} stresses the importance of knowing examples of colorings on finite sets belonging to $\UR$. 
Our first example is trivial but we include it for future reference. 

\begin{ex}
\label{constant}
Any constant coloring belongs to $\UR$.
\end{ex}

The next result provides a general method to extend any coloring on a finite set to a reconstructible one. It will be used several times in the sequel. 

\begin{prop}
\label{propertyE}
Let $\varphi_0$ be any coloring of the pairs of $F=\{x,y,w,z\}$. Let $a$ and $b$ be two elements not in $F$. The  coloring $\varphi$ on $F\cup \{a,b\}$ extending  $\varphi_0$ as in Figure \ref{figpropertyE}  is reconstructible (where the colors between the elements of $F$ are not drawn).

\bigskip

\begin{minipage}{.5\textwidth}
\centering
\begin{tikzpicture}[scale=0.7]
\filldraw[black] (0,0) circle (1.5pt) node[anchor=east] {\scriptsize$x$};
\filldraw[black] (2,4) circle (1.5pt) node[anchor=east] {\scriptsize$a$};
\filldraw[black] (4,4) circle (1.5pt) node[anchor=west] {\scriptsize$b$};
\filldraw[black] (2,0) circle (1.5pt) node[anchor=east] {\scriptsize$y$};
\filldraw[black] (4,0) circle (1.5pt) node[anchor=east] {\scriptsize$z$};
\filldraw[black] (6,0) circle (1.5pt) node[anchor=west] {\scriptsize$w$};

\draw[thick] (0,0) -- (2,4);
\draw[thick] (0,0) -- (4,4);
\draw[thick] (2,4) -- (4,4);
\draw[thick] (2,0) -- (2,4);
\draw[thick] (2,0) -- (4,4);
\draw[thick] (2,4) -- (4,0);
\draw[thick] (4,4) -- (4,0);
\draw[thick] (4,4) -- (6,0);
\draw[thick] (2,4) -- (6,0);
%
 \end{tikzpicture}
 \captionof{figure}{Partial drawing of $\varphi$}
 \label{figpropertyE}
\end{minipage}
\begin{minipage}{.5\textwidth}
\centering
\begin{tikzpicture}[scale=0.7]
\filldraw[black] (10,0) circle (1.5pt) node[anchor=east] {\scriptsize$x$};
\filldraw[black] (12,4) circle (1.5pt) node[anchor=east] {\scriptsize$a$};
\filldraw[black] (14,4) circle (1.5pt) node[anchor=west] {\scriptsize$b$};
\filldraw[black] (12,0) circle (1.5pt) node[anchor=east] {\scriptsize$y$};
\filldraw[black] (14,0) circle (1.5pt) node[anchor=east] {\scriptsize$z$};
\filldraw[black] (16,0) circle (1.5pt) node[anchor=west] {\scriptsize$w$};

\draw[mygray,thick] (10,0) -- (12,4);
\draw[mygray,thick] (10,0) -- (14,4);
\draw[mygray,thick] (12,4) -- (14,4);
\draw[mygray,thick] (12,0) -- (12,4);
\draw[mygray,thick] (12,0) -- (14,4);
\draw[mygray,thick] (12,4) -- (14,0);
\draw[mygray,thick] (14,4) -- (14,0);
\draw[mygray,thick] (14,4) -- (16,0);
\draw[mygray,thick] (12,4) -- (16,0);
 
\end{tikzpicture}
\captionof{figure}{Partial drawing of $\psi$}
\label{figpropertyEpsi}
\end{minipage}

\end{prop}

\bigskip

\begin{proof}
Let $X=F\cup \{a,b\}$ and  $\psi$ be a coloring of $[X]^2$ such that $\h(\varphi)=\h(\psi)$. Suppose  there is  $\{u,v\}\in [X]^2$ such that $\varphi(\{u,v\})=1-\psi(\{u,v\})$. 
We will show that $\varphi=1-\psi$. 
We will assume that $u=x$ and $v=y$. A completely analogous argument works for the other cases. 
Notice that $\{a,b,x\},\{a,b,y\}, \{a,b,z\}, \{a,b,w\}\in\h(\varphi)=\h(\psi)$. Let $i=\varphi(\{x,a\})$.  Thus $\psi$  looks as depicted in Figure \ref{figpropertyEpsi}. Again the colors between elements of  $F$ are not drawn.

We consider two cases:

{\em Case 1:} Suppose $\varphi(\{x,y\})=i$.  It follows that  $\{x,y,a\}\in \h(\varphi)=\h(\psi)$, and therefore $\psi(\{a,z\})=\psi(\{a,b\})=\psi(\{a,w\})=\psi(\{a,y\})=\psi(\{x,y\})=1-i$. Now notice that $\varphi(\{z,w\})=\psi(\{z,w\})$ would imply $\{a,z,w\}\in\h(\varphi)\triangle\h(\psi)$ which is a contradiction. It follows that $\varphi(\{z,w\})=1-\psi(\{z,w\})$.

\medskip

{\em Case 2:} Suppose $\varphi(\{x,y\})=1-i$.  Then $\{x,y,a\}\notin\h(\varphi)=\h(\psi)$. But, $\psi(\{a,x\})=\psi(\{a,b\})=\psi(\{a,y\})$, thus $\psi(\{a,x\})\neq\psi(\{x,y\})=i$ and therefore $\psi(\{a,w\})=\psi(\{a,z\})=\break\psi(\{a,x\})=1-i$. Then, we argue as in the previous case to see that $\varphi(\{z,w\})\neq\psi(\{z,w\})$.

In either case, we have that $\varphi=1-\psi$. 
\end{proof}

From the previous result  we get the following more general fact. 

\begin{prop}
\label{extensionUR}
Let $\varphi$ be a coloring on $X$ and $a,b\not\in X$. Then, there is a coloring $\psi$ on $X\cup\{a,b\}$ such that $\varphi\subset \psi$ and $\psi\in \UR$. 
\end{prop}

\begin{proof} Define $\psi$ on $X\cup\{a,b\}$ by $\psi(\{a,b\})=\psi(\{a,x\})=\psi(\{b,x\})=1$ for all $x\in X$, and $\varphi\subset \psi$. From Proposition \ref{propertyE} we get that  $\psi $ satisfies the hypothesis of Proposition \ref{4suffices}, hence $\psi\in \UR$.\end{proof}

As an application of Proposition \ref{4suffices} we have the following result about a  coloring  on binary sequences. 

\begin{prop}
\label{binarytree}
The coloring associated to the extension ordering on binary sequences is reconstructible. 
\end{prop}

\begin{proof}
Let $\varphi$ be the coloring associated to the extension ordering on $2^{<\omega}$, i.e., $\varphi(\{x,y\})=1$ if and only if $y$ is an extension of $x$.
We first show the result for the restriction of $\varphi$ to $X=2^{\leq 3}$. This coloring looks as despicted in Figure \ref{figbinarytree}, where only some $1$-edges are drawn. 

\bigskip

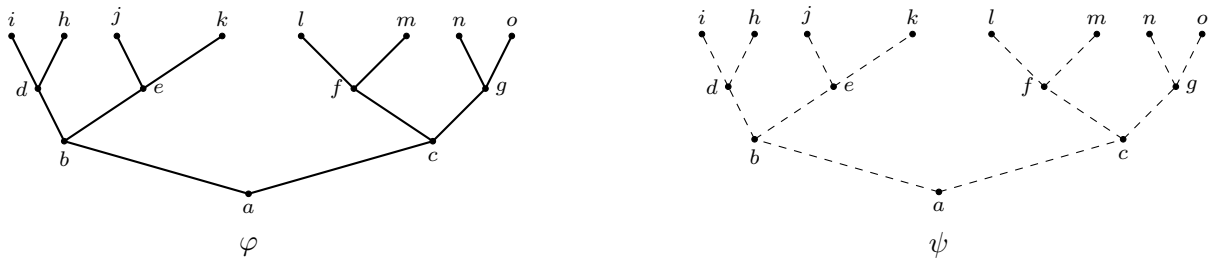
\hspace*{-1cm}
\begin{center}
\begin{tikzpicture}[scale=0.8]
\filldraw[black] (0,0) circle (1.5pt) node[anchor=north] {\scriptsize$a$};
\filldraw[black] (-3.5,1) circle (1.5pt) node[anchor=north] {\scriptsize$b$};
\filldraw[black] (3.5,1) circle (1.5pt) node[anchor=north] {\scriptsize$c$};
\filldraw[black] (-4,2) circle (1.5pt) node[anchor=east] {\scriptsize$d$};
\filldraw[black] (-2,2) circle (1.5pt) node[anchor=west] {\scriptsize$e$};
\filldraw[black] (2,2) circle (1.5pt) node[anchor=east] {\scriptsize$f$};
\filldraw[black] (4.5,2) circle (1.5pt) node[anchor=west] {\scriptsize$g$};
\filldraw[black] (-3.5,3) circle (1.5pt) node[anchor=south] {\scriptsize$h$};
\filldraw[black] (-4.5,3) circle (1.5pt) node[anchor=south] {\scriptsize$i$};
\filldraw[black] (-2.5,3) circle (1.5pt) node[anchor=south] {\scriptsize$j$};
\filldraw[black] (-0.5,3) circle (1.5pt) node[anchor=south] {\scriptsize$k$};
\filldraw[black] (1,3) circle (1.5pt) node[anchor=south] {\scriptsize$l$};
\filldraw[black] (3,3) circle (1.5pt) node[anchor=south] {\scriptsize$m$};
\filldraw[black] (4,3) circle (1.5pt) node[anchor=south] {\scriptsize$n$};
\filldraw[black] (5,3) circle (1.5pt) node[anchor=south] {\scriptsize$o$};

  \draw[thick] (0,0) -- (-3.5,1);
  \draw[thick] (-3.5,1) -- (-4,2);
  \draw[thick] (-3.5,1) -- (-2,2);
  \draw[thick] (-4,2) -- (-4.5,3);
  \draw[thick] (-4,2) -- (-3.5,3);
  \draw[thick] (-2,2) -- (-2.5,3);
 \draw[thick] (-2,2) -- (-0.5,3);
 \draw[thick] (0,0) -- (3.5,1);
\draw[thick] (3.5,1) -- (2,2);
\draw[thick] (3.5,1) -- (4.5,2);
\draw[thick] (2,2) -- (1,3);
\draw[thick] (2,2) -- (3,3);
\draw[thick] (4.5,2) -- (4,3);
\draw[thick] (4.5,2) -- (5,3);
 
\end{tikzpicture}
\captionof{figure}{Partial drawing of $\varphi$}
\label{figbinarytree}
\end{center}

\bigskip

We show that $\varphi\in \UR$. Let $\psi$ be a coloring of $[X]^2$ such that $\h(\varphi)=\h(\psi)$. Notice that every branch and every antichain is homogeneous (for both colorings). Suppose that $\varphi(\{x,y\})=1-\psi(\{x,y\})$ for some $\{x,y\}\in[X]^2$. We need to show that $\varphi=1-\psi$.  We consider the case $x=d$ and $y=b$, the other cases are similar. 
Then all branches starting on $i$, $h$, $j$ or $k$ are of color $1$ for $\varphi$ and of color $0$ for $\psi$. Since $\{i,h,d\}$ is not homogeneous, then $\psi(\{i,h\})=1$. Therefore $\{i,h,j,k,l,m,n,o\}$ is $1$-homogeneous for $\psi$. 
Since $\{l,m,f\}$ is not homogeneous, then $\psi(\{l,f\})= 0$ or $\psi(\{m,f\})=0$. In either case, we get that all branches starting from $l$, $m$, $n$ or $o$ are all $0$-homogeneous for $\psi$. As before, we conclude that $\{d,e,f,g\}$ and $\{b,c\}$ are 1-homogeneous for $\psi$. This shows that $\varphi=1-\psi$. 

Now we finish the proof of the proposition. 
To see that $\varphi\in \UR$, we use Proposition \ref{4suffices}. Let $F\subset 2^{<\omega}$ be a set with at most 4 elements. It is easy to verify that $\langle F,\varphi|_F\rangle$ is isomorphic (as a graph) to a subset of $\langle X,\varphi|_ X\rangle$. From the result above,  $\varphi|_X\in \UR$ and we are done. 
\end{proof}

The following examples will be needed later in the paper. 

\begin{ex}
\label{particion}
Let $X=\{0,1,2,3,4,5\}$ and consider the partition of $X$ given by $\{0,1,2\}$, $\{3,4\}$ and $\{5\}$. 
Let $\varphi$ be the coloring associated to  this partition. It is depicted in Figure \ref{figpartition},  where we only draw the pairs with color 1, i.e. those $\{x,y\}$ which are  a subset of a part of the partition. 

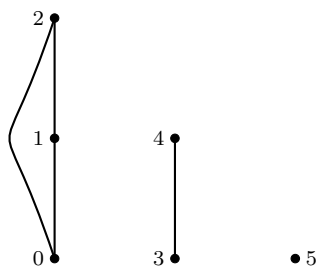
\begin{center}
\begin{tikzpicture}[scale=1.6]
\draw[thick] (0,0) -- (0,1);
\draw[thick] (0,1) -- (0,2);
\draw[thick] (1,0) -- (1,1);
\draw[thick] (0,0) .. controls (-0.5,1.5) and (-0.5,0.5) .. (0,2);

\filldraw[black] (0,0) circle (1pt) node[anchor=east] {\scriptsize $0$};
\filldraw[black] (0,1) circle (1pt) node[anchor=east] {\scriptsize $1$};
\filldraw[black] (0,2) circle (1pt) node[anchor=east] {\scriptsize $2$};

\filldraw[black] (1,0) circle (1pt) node[anchor=east] {\scriptsize $3$};
\filldraw[black] (1,1) circle (1pt) node[anchor=east] {\scriptsize $4$};
\filldraw[black] (2,0) circle (1pt) node[anchor=west] {\scriptsize $5$};

\end{tikzpicture}
\captionof{figure}{}
\label{figpartition}
\end{center}

We claim that $\varphi|_Y\in \UR$ for every $Y\subseteq X$. We show it for  $X=Y$, the rest is similar. Let $\psi$ be a coloring on $X$ such that $\h(\varphi)=\h(\psi)$. Notice that $\{0,4,5\}$, $\{0,3,5\}$  and $\{2,4,5\}$ are $\varphi$-homogeneous and $\{0,1,3\}$ and $\{3,4,5\}$ are  not $\varphi$-homogeneous. Since $\h(\varphi)=\h(\psi)$, $\psi(\{0,1\})= \psi(\{3,4\})=1-\psi(\{0,3\})$. Thus, $\psi$ is either $\varphi$ of $1-\varphi$. 
\end{ex}

\begin{prop}
\label{particion2}
Let $\varphi$ be a coloring on a set $F=\{a,b,c,d, e\}$, and let $G=\{x,y,z\}$ be disjoint from $F$. Let  $X=F\cup G$, and  $\psi$ be the extension of $\varphi$ to  $X$ as depicted in Figure \ref{figpartition2}, where  we only draw the pairs $\{u,v\}$ of  color 1 with $u\in F$and $v\in G$. Then $\psi\in \UR$.

\begin{center}
\begin{tikzpicture}[scale=1.6]
\draw[thick] (-1,0) -- (0,1);
\draw[thick] (0,0) -- (0,1);
\draw[thick] (1,0) -- (1,1);
\draw[thick] (2,0) -- (2,1);

\filldraw[black] (-1,0) circle (1pt) node[anchor=north] {\scriptsize$a$};
\filldraw[black] (0,0) circle (1pt) node[anchor=north] {\scriptsize$b$};
\filldraw[black] (0,1) circle (1pt) node[anchor=south] {\scriptsize$x$};
\filldraw[black] (1,0) circle (1pt) node[anchor=north] {\scriptsize$c$};
\filldraw[black] (1,1) circle (1pt) node[anchor=south] {\scriptsize$y$};
\filldraw[black] (2,0) circle (1pt) node[anchor=north] {\scriptsize$d$};
\filldraw[black] (2,1) circle (1pt) node[anchor=south] {\scriptsize $z$};
\filldraw[black] (3,0) circle (1pt) node[anchor=north] {\scriptsize$e$};
\end{tikzpicture}
\captionof{figure}{Partial drawing of $\psi$.}
\label{figpartition2}

\end{center}
\end{prop}
\begin{proof}
Let $\rho$ be a coloring on $X$ such that $\h(\psi)=\h(\rho)$. We assume without lost of generality that $\rho(\{x,a\})=\psi(\{x,a\})=1$, and we prove that $\rho=\psi$. Using the same kind of arguments as in Example \ref{particion}, it is easy to verify that $\rho(\{u,v\})=\psi(\{u,v\})$ for every 
$u\in F$ and $v\in G$, and  also for $u, v\in G$. So, it  remains to show that $\rho $ also extends $\varphi$. Indeed, given $u, v\in F$, there is $w\in \{x,y,z\}$ such that $\psi(\{w, u\})=\psi(\{w,v\})=0$. Thus, $\rho(\{w, u\}) =\rho(\{w,v\})=0$.  We need to show that  $\psi(\{ u,v\})=\rho(\{u,v\})$. Suppose otherwise, $\psi(\{ u,v\})\not=\rho(\{u,v\})$. Then, $\{u,v,w\}\in \h(\rho)\triangle \h(\psi)$, a contradiction.
\end{proof}

\begin{ex}
\label{recmax}
The coloring $\varphi$  on $\{0,1,2,3,4,5\}$ depicted in Figure \ref{figrecmax}  is reconstructible.

\begin{center}
\begin{tikzpicture}[scale=1.2]
\draw[thick] (0,0) -- (0,1);
\draw[thick] (0,1) -- (0,2);

\draw[thick] (0,0) .. controls (-0.5,1.5) and (-0.5,0.5) .. (0,2);

\draw[thick] (2,0) -- (2,1);
\draw[gray] (0,0) -- (2,0);
\draw[thick] (0,2) -- (1,-1);
\draw[gray] (0,2) -- (2,1);
\draw[gray] (0,2) -- (2,0);
\draw[gray] (0,0) -- (2,1);
\draw[gray] (0,1) -- (2,1);
\draw[gray] (0,1) -- (2,0);
\draw[thick] (0,0) -- (1,-1);
\draw[thick] (0,1) -- (1,-1);
\draw[thick] (2,1) -- (1,-1);
\draw[thick] (2,0) -- (1,-1);

\filldraw[black] (0,0) circle (1pt) node[anchor=east] {\scriptsize $1$};
\filldraw[black] (0,1) circle (1pt) node[anchor=east] {\scriptsize $2$};
\filldraw[black] (0,2) circle (1pt) node[anchor=east] {\scriptsize $3$};

\filldraw[black] (2,1) circle (1pt) node[anchor=west] {\scriptsize $5$};
\filldraw[black] (2,0) circle (1pt) node[anchor=west] {\scriptsize $4$};
\filldraw[black] (1,-1) circle (1pt) node[anchor=north] {\scriptsize $0$};

\end{tikzpicture}
\captionof{figure}{Coloring $\varphi$}
\label{figrecmax}

\end{center}
Let $\psi$ be a reconstruction of $\varphi$, i.e. $\h(\varphi)=\h(\psi)$. Notice that any  homogeneous set is a subset of either $H_1=\{0,1,2,3\}$ or $H_2=\{0,4,5\}$.  It is easy to check that if $\psi$ gives the same color, say black, to $H_1$ and $H_2$, then $\psi=\varphi$.  So, suppose $\psi$ gives to $H_1$ and $H_2$ color black and gray, respectively. Then one has to consider  two completely analogous cases depending on  whether  $\psi(\{1,4\}$ is black or gray. 
Suppose it is black. Since $\psi(\{1,2\})$ is black and $\{1,2,4\}$ is not homogeneous, $\psi(\{2,4\})$ is gray. Since $\{2,4,5\}$ is not homogeneous, $\psi(\{2,5\})$ is black. Analogously, one conclude that $\psi(\{1,5\})$ is gray.  Since $\{3,4,5\}$ is not homogeneous, $\{3,5\}$ must be black. On the other hand, since $\{2,3,5\}$ is not homogeneous, $\{3,5\}$ must be gray. A contradiction.

\end{ex}

\subsection{Properties $E_0$ and $E_1$}\hfill

Now we introduce a property for a coloring stronger than being in $\UR$. It was motivated by the extension property of the random graph. Given $i\in\{0,1\}$, we say that a coloring $\varphi:[\N]^2\longrightarrow 2$ has the {\em property $E_i$} if for every finite set $F\subset \N$ there is $z\in\N\setminus F$ such that $\varphi(\{z,x\})=i$ for every $x\in F$. 

It is clear that if $\varphi$ has the property $E_0$ then $1-\varphi$ has the property $E_1$. So, for our reconstruction problem, we could only work with either $E_0$ or $E_1$.  The random graph clearly has the property $E_i$, for $i\in\{0,1\}$. Another example is the following.

\begin{prop}
\label{general-Sierpinski}
Let $R\subseteq \N\times \N$ be a strict linear ordering on $\N$. Let $\varphi_R$ be defined by 
$\varphi_R (\{n,m\}_<)=1$ if and only if $(n,m)\in R$. If $\langle \N, R\rangle$ does not have a maximal (resp. a minimal) element, then  $\varphi_R$ has the property $E_1$ (resp. $E_0$).
\end{prop}

\begin{proof}
Suppose $\langle \N, R\rangle$ does not have a maximal element. Let $F\subseteq \N$ be a finite set. Then, there is $z\in \N$ such that $(x,z)\in R$ for all $x\in F$. Thus, $\varphi_R(\{z,x\})=1$ for all $x\in F$. 
\end{proof}

\begin{prop}
\label{EiUR}
Every coloring with property $E_i$, $i\in\{0,1\}$, belongs to $\UR$.
\end{prop}

\begin{proof} We will use Proposition \ref{4suffices}. Let $\varphi$ be a coloring with  the property  $E_i$. Let $F=\{x,y,z,w\}$ be a subset of $\N$. 
By the property $E_i$, there is $a\in\N\setminus \{x,y,z,w\}$ such that 
\[
\varphi(\{a,x\})=\varphi(\{a,y\})=\varphi(\{a,z\})=\varphi(\{a,w\})=i
\]
and there is $b\in\N\setminus \{x,y,z,w,a\}$ such that 
\[
\varphi(\{b,x\})=\varphi(\{b,y\})=\varphi(\{b,z\})=\varphi(\{b,w\})=\varphi(\{b,a\})=i.
\]
Let $X=F\cup\{a,b\}$. Now observe that $\langle X,\varphi|_X\rangle$ is isomorphic to the graph in Proposition \ref{propertyE}. Thus, $\varphi|_X\in \UR$ and we are done. 
\end{proof}

We will now see that any coloring with the property $E_i$ provides infinitely many reconstructible colorings obtained by making finite changes to the original one.

Let $\varphi\in 2^{[\N]^2}$ and $a\subset [\N]^2$ be a finite set. Let $\varphi_a\in 2^{[\N]^2}$ be defined by $\varphi_a^{-1}(1)=a\triangle \varphi^{-1}(1)$. In other words, $\varphi_a(\{x,y\})=\varphi(\{x,y\})$ if $\{x,y\}\notin a$; and $\varphi_a(\{x,y\})=1-\varphi(\{x,y\})$ if $\{x,y\}\in a$, for every $\{x,y\}\in[\N]^2$. Such colorings are the {\em finite changes} of $\varphi$.

\begin{prop}
\label{finitechangesEi} 
Let $\varphi$ be a coloring on $\N$ and  $a\subset [\N]^2$ a finite set. If $\varphi$ has the property $E_i$,  for $i\in\{0,1\}$, then $\varphi_a$ has the property $E_i$. 
\end{prop}

\begin{proof} Let us fix $i\in\{0,1\}$ and assume that  $\varphi$  has  property $E_i$. Let  $F\subset\N$ be a finite set, and consider $G=F\cup \{w: \{w,z\}\in a \;\text{for some $z\in \N$}\}$.  By the  property $E_i$ of $\varphi$, there is $z\in\N\setminus G$ such that  $\varphi(\{z,x\})=i$ for every $x\in G$. Given $x\in F$, we have $\{z,x\}\notin a$ since $z\not\in G$. Thus, $\varphi_a(\{z,x\})=\varphi(\{z,x\})=i$.\end{proof}

\begin{coro}\label{finitechanges} The finite changes of the following colorings are reconstructible:
\begin{enumerate}[label=(\roman*)]
    \item Constant colorings on $\N$. 
    
     \item The  random graph.
    
    \item The  Sierpi\'{n}ski's coloring.
    
    \item  $\varphi_R$, for $R\subseteq \N\times \N$ a linear ordering on $\N$ without maximal or minimal element.
\end{enumerate}
\end{coro}

\section{ Non-reconstructible colorings}\label{nonrecoverable}

In this section we analyze non-reconstructible colorings. We start by showing a condition that implies non reconstructibility and which is  used in almost all examples   presented. We also show that any coloring can be extended to a non- reconstructible one (Proposition \ref{extentioninnotR}).

\begin{defi} For a coloring $\varphi$ on $X$ and $x, y\in X$,  we say that an edge $\{x,y\}$ is {\em critical for $\varphi$}, if 
$\varphi(\{x,z\})=1-\varphi(\{y,z\})$, for all $z\in X\setminus\{x,y\}$.
\end{defi}

The following simple observation gives a very useful criterion to show non reconstructibility.

\begin{prop}
\label{criterio-no-UR}
Let $\varphi$ be a coloring on $X$ with $|X|\geq 3$.  If $\varphi$ has a critical pair, then $\varphi$ is non-reconstructible.
\end{prop}

\proof Let $\{x,y\}$ be a critical pair for $\varphi$. Define $\psi:[X]^2\longrightarrow 2$ by $\psi(\{w,z\})=\varphi(\{w,z\})$ if $\{w,z\}\neq\{x,y\}$ and $\psi(\{x,y\})=1-\varphi(\{x,y\})$.  Notice that $\{x,y\}\not\subseteq H$ for any $H\in \h(\varphi)\cup \h(\psi)$. Therefore $\h(\varphi)=\h(\psi)$ and $\psi$ witnesses that  $\varphi\not\in\UR$.
\endproof

The condition of being a critical pair is  stronger than just requiring that the pair is not contained in a homogeneous set. For instance, let $\varphi$ be  a constant coloring over $\N$ and consider the finite change $\varphi_a$  of $\varphi$ with $a=\{0,1\}$. Then $\varphi_a$ is reconstructible and $\{0,1\}$ is not contained in any $\varphi_a$-homogeneous sets. 

Now we give our first example of a non-reconstructible coloring, which is  the prototype of such colorings. 

\begin{ex}\label{nonrecovered}
Consider a partition of $\N$ into two infinite sets, for instance,  let $A_0$ be the set of even numbers and $A_1$ be the set of  odd numbers. Let $\varphi$ be the coloring associated to this partition, i.e.,  $\varphi(\{x,y\})=1$ if and only if $\{x,y\}\subseteq A_i$ for some $i$. Then, the pair $\{0,1\}$ is critical for $\varphi$, thus by Proposition \ref{criterio-no-UR}, $\varphi\in\neg\UR$.

Moreover, given any nonempty set $B\subseteq \N$ consider the coloring $\varphi_B:[\N]^2\longrightarrow 2$ given by 
$\varphi_B(\{x,y\})=\varphi(\{x,y\})$ if $\{x,y\}\neq\{2n, 2n+1\}$ for any $n\in B$; and $\varphi_B(\{2n, 2n+1\})=1$ for all $n\in B$. Then, $\varphi$ and $\varphi_B$ have the same homogeneous sets. 

Below we depict the  coloring $\varphi$ (see Figure \ref{fignonrec}) and a non-trivial reconstruction $\psi$ (see Figure \ref{fignonrec2}) of it for the case of a partition of the set $\{0,1,2,3,4,5\}$. 

\bigskip

\begin{minipage}{.5\textwidth}
\centering
\begin{tikzpicture}[scale=1.6]
\draw[thick] (0,0) -- (0,1);
\draw[thick] (0,1) -- (0,2);
\draw[thick] (1,0) -- (1,1);
\draw[thick] (1,1) -- (1,2);
\draw[thick] (0,0) .. controls (-0.5,1.5) and (-0.5,0.5) .. (0,2);
\draw[thick] (1,0) .. controls (1.5,1.5) and (1.5,0.5) .. (1,2);
\draw[gray] (0,0) -- (1,0);
\draw[gray] (0,0) -- (1,1);
\draw[gray] (0,0) -- (1,2);
\draw[gray] (0,1) -- (1,0);
\draw[gray] (0,1) -- (1,1);
\draw[gray] (0,1) -- (1,2);
\draw[gray] (0,2) -- (1,0);
\draw[gray] (0,2) -- (1,1);
\draw[gray] (0,2) -- (1,2);

\filldraw[black] (0,0) circle (1pt) node[anchor=east] {\scriptsize $0$};
\filldraw[black] (0,1) circle (1pt) node[anchor=east] {\scriptsize $2$};
\filldraw[black] (0,2) circle (1pt) node[anchor=east] {\scriptsize $4$};

\filldraw[black] (1,0) circle (1pt) node[anchor=west] {\scriptsize $1$};
\filldraw[black] (1,1) circle (1pt) node[anchor=west] {\scriptsize $3$};
\filldraw[black] (1,2) circle (1pt) node[anchor=west] {\scriptsize $5$};

\end{tikzpicture}
\captionof{figure}{Coloring $\varphi$}
\label{fignonrec}

\end{minipage}
\begin{minipage}{.5\textwidth}
\centering
\begin{tikzpicture}[scale=1.6]
\draw[thick] (0,0) -- (0,1);
\draw[thick] (0,1) -- (0,2);
\draw[thick] (1,0) -- (1,1);
\draw[thick] (1,1) -- (1,2);
\draw[thick] (0,0) -- (1,0);
\draw[thick] (0,0) .. controls (-0.5,1.5) and (-0.5,0.5) .. (0,2);
\draw[thick] (1,0) .. controls (1.5,1.5) and (1.5,0.5) .. (1,2);
\draw[gray] (0,0) -- (1,1);
\draw[gray] (0,0) -- (1,2);
\draw[gray] (0,1) -- (1,0);
\draw[gray] (0,1) -- (1,1);
\draw[gray] (0,1) -- (1,2);
\draw[gray] (0,2) -- (1,0);
\draw[gray] (0,2) -- (1,1);
\draw[gray] (0,2) -- (1,2);

\filldraw[black] (0,0) circle (1pt) node[anchor=east] {\scriptsize $0$};
\filldraw[black] (0,1) circle (1pt) node[anchor=east] {\scriptsize $2$};
\filldraw[black] (0,2) circle (1pt) node[anchor=east] {\scriptsize $4$};

\filldraw[black] (1,0) circle (1pt) node[anchor=west] {\scriptsize $1$};
\filldraw[black] (1,1) circle (1pt) node[anchor=west] {\scriptsize $3$};
\filldraw[black] (1,2) circle (1pt) node[anchor=west] {\scriptsize $5$};

\end{tikzpicture}
\captionof{figure}{Coloring $\psi$}
\label{fignonrec2}

\end{minipage}
\end{ex}

\bigskip

\bigskip

We  present below a characterization of colorings admitting a critical pair. For that purpose,  we introduce a function on colorings. Let $X$ be a set with $3\leq |X|\leq \aleph_0$. For each coloring $\varphi\not\in\UR$ on $X$, let 

\[
r(\varphi)=\min\{ |\{\{x,y\}\in [X]^2: \varphi(\{x,y\})\neq \psi(\{x,y\})\}|:\psi\in 2^{[X]^2},\; \h(\psi)=\h(\varphi), \psi\neq\varphi, \psi\neq 1-\varphi\}.
\]
\medskip

For convenience, let $r(\varphi)=0$ if $\varphi\in\UR$. 
Notice $0\leq r(\varphi)\leq \aleph_0$.  

\begin{lema}
\label{rnoes2}
Let $\varphi$ be a coloring on a countable set with at least 3 elements.
\begin{itemize}
\item[(i)] If $\varphi$ has a critical pair, then $r(\varphi)=1$.
\item[(ii)] $r(\varphi)\neq 2$ for every $\varphi$. 
\end{itemize}
\end{lema}

\begin{proof} 

(i) By Proposition \ref{criterio-no-UR}, $\varphi\not\in \UR$ and the result follows from the proof of Proposition \ref{criterio-no-UR}.

(ii) If $\varphi\in \UR$, then $r(\varphi)=0$.  Let $\varphi\in\neg \UR$, and suppose $r(\varphi)=2$ to get a contradiction. Let $\psi$ be a reconstruction of 
$\varphi$ such that  
\begin{equation}
\label{d2}
|\{\{x,y\}: \varphi(\{x,y\})\neq \psi(\{x,y\})\}|=2.
\end{equation}
Let $\{x,y\}$ be such that $\varphi(\{x,y\})\neq \psi(\{x,y\})$. By (i), $\varphi$ does not have a critical pair. Since $\{x, y\}$ is not critical for $\varphi$, there is $z\not\in\{x,y\}$ such that $\varphi(\{x,z\})=\varphi(\{y,z\})$.  We claim that  $\{x,y,z\}\not\in \h(\varphi)$. Suppose not and let $i$ be its $\varphi$-color. Then,  $\{x,y,z\}$ would be a $\psi$-homogeneous set of color $1-i$, which contradicts  (\ref{d2}). Thus 
\begin{equation}
\label{uno}
\varphi(\{y,z\})=\varphi(\{x,z\})=1-\varphi(\{x,y\})=\psi(\{x,y\}).
\end{equation}
Since $\{x,y,z\}$ is not $\psi$-homogeneous,  we assume, without lost of generality,  that
\begin{equation}
\psi(\{x,z\})=1-\psi(\{x,y\}).
\end{equation}
Notice that $\varphi(\{x,z\})\neq \psi(\{x,z\})$ and, by (\ref{d2}), $\varphi$ and $\psi$ agree on any pair different from $\{x,z\}$ and $\{x,y\}$.  Thus 
\begin{equation}
\psi(\{y,z\})=\varphi(\{y,z\}). 
\end{equation}
 Since $\{x,z\}$ is not critical for $\varphi$, there is $w\not\in\{x,z\}$ such that 
 $\varphi(\{x,w\})=\varphi(\{z,w\})$. From (\ref{uno}), $w\neq y$.  By (\ref{d2}), $\varphi(\{x,w\})=\psi(\{x,w\})=\psi(\{z,w\})$. It is easy to verify that $\{x,w,z\}\in\h(\varphi)\triangle\h(\psi)$, a contradiction. 
\end{proof}

\begin{teo}
\label{criterio-no-UR2}
Let $\varphi$ be a  coloring on  a set $X$ with $|X|\geq 3$. The following are 
equivalent.

\begin{itemize}
\item[(i)]  There is a critical pair for $\varphi$. 

\item[(ii)]  $r(\varphi)=1$.

\item[(iii)] There is  a coloring $\psi$ and  $x\in X$ such that $\h(\varphi)=\h(\psi)$, $\varphi\neq\psi$ and $\varphi|_{X\setminus\{x\}}=\psi |_{X\setminus\{x\}}$.

\end{itemize}
\end{teo}

\begin{proof}   
$(i) \Rightarrow (ii)$.  By Lemma \ref{rnoes2}.

$(ii) \Rightarrow (iii)$.  Obvious.

$(iii) \Rightarrow (i)$. 
Let  $\psi$ and $x$ be as in the hypothesis of (iii). Towards a contradiction, suppose there are no critical pairs for $\varphi$. Let  $y\in X\setminus\{x\}$ be such that $\varphi(\{x,y\})\neq\psi(\{x,y\})$. Since $\{x,y\}$ is not critical for $\varphi$, there is $z\not\in \{x,y\}$ such that $\varphi(\{x,z\})=\varphi(\{y,z\})$.
Since  $\varphi|_{X\setminus\{x\}}=\psi |_{X\setminus\{x\}}$,  $\psi(\{y,z\})=\varphi(\{y,z\})$. Hence, $\{x,y,z\}\not\in\h(\varphi)=\h(\psi)$, otherwise $\psi(\{x,y\})=\psi(\{y,z\})=\varphi(\{y,z\})=\varphi(\{x,y\})$, a contradiction.  Then, 
$\varphi(\{x,z\})\neq\psi(\{x,z\})$. By Lemma \ref{rnoes2}, $r(\varphi)\geq 3$, thus there is $\{u,w\}\in [X]^2$ with $\{u,w\}$ different from $\{x,y\}, \{x,z\}$ such that $\varphi(\{u,w\})\neq\psi(\{u,w\})$. 
As $\varphi|_{X\setminus\{x\}}=\psi |_{X\setminus\{x\}}$, we assume that $u=x$, i.e. $\varphi(\{x,w\})\neq\psi(\{x,w\})$. There are two cases to be considered:
(a) Suppose $\varphi(\{x,w\})= \varphi(\{x,y\})$. Then,   $\{x,y,w\}\in \h(\varphi)\triangle\h(\psi)$. 
(b) Suppose $\varphi(\{x,w\})=1- \varphi(\{x,y\})$. Then,   $\{x,w,z\}\in \h(\varphi)\triangle\h(\psi)$.
In both cases we get a  contradiction with $\h(\varphi)=\h(\psi)$. 
\end{proof}

We know very little about the function $r$.

\begin{question}
 Is there a coloring $\varphi$ such that $r(\varphi)=\aleph_0$?
\end{question}

There are non-reconstructible colorings without a critical pair, as we show next. However, we do not know a method to construct colorings in $\neg\UR$ without critical pairs.

\begin{ex}
\label{sin-pareja-critica}
The colorings $\varphi$ and $\varphi'$ depicted below (Figures \ref{fignoncritical1} and \ref{fignoncritical2}) are non-reconstructible and  do not have a critical pair. Colorings $\psi$ and $\psi'$ (Figures \ref{fignoncritical1b} and \ref{fignoncritical2b}) are, respectively, a non-trivial  reconstruction of $\varphi$ and $\varphi'$. 
\medskip

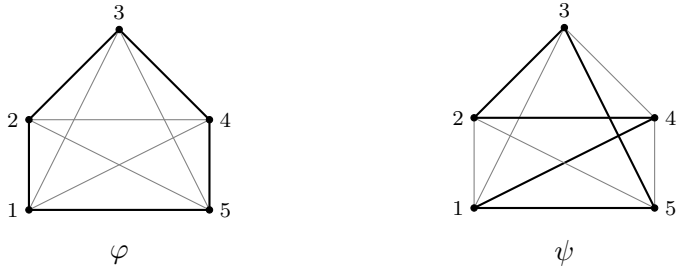
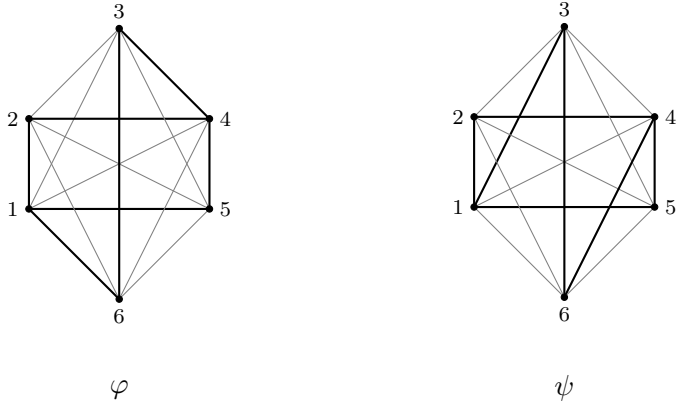
\begin{minipage}{.5\textwidth}
\centering
\begin{tikzpicture}[scale=1.2]
\draw[thick] (0,0) -- (0,1);
\draw[thick] (2,0) -- (2,1);
\draw[thick] (0,0) -- (2,0);
\draw[thick] (0,1) -- (1,2);
\draw[thick] (1,2) -- (2,1);

\draw[mygray,thick] (0,0) -- (1,2);
\draw[mygray,thick] (0,0) -- (2,1);
\draw[mygray,thick] (0,1) -- (2,1);
\draw[mygray,thick] (0,1) -- (2,0);
\draw[mygray,thick] (1,2) -- (2,0);

\filldraw[black] (0,0) circle (1pt) node[anchor=east] {\scriptsize $1$};
\filldraw[black] (0,1) circle (1pt) node[anchor=east] {\scriptsize $2$};
\filldraw[black] (1,2) circle (1pt) node[anchor=south] {\scriptsize $3$};
\filldraw[black] (2,1) circle (1pt) node[anchor=west] {\scriptsize $4$};
\filldraw[black] (2,0) circle (1pt) node[anchor=west] {\scriptsize $5$};
\end{tikzpicture}
\captionof{figure}{$\varphi$}
\label{fignoncritical1}

\end{minipage}
\begin{minipage}{.5\textwidth}
\centering
\begin{tikzpicture}[scale=1.2]

\draw[mygray,thick] (0,0) -- (0,1);
\draw[mygray,thick] (2,0) -- (2,1);
\draw[thick] (0,0) -- (2,0);
\draw[thick] (0,1) -- (1,2);
\draw[mygray,thick] (1,2) -- (2,1);

\draw[mygray,thick] (0,0) -- (1,2);
\draw[thick] (0,0) -- (2,1);
\draw[thick] (0,1) -- (2,1);
\draw[mygray,thick] (0,1) -- (2,0);
\draw[thick] (1,2) -- (2,0);

\filldraw[black] (0,0) circle (1pt) node[anchor=east] {\scriptsize $1$};
\filldraw[black] (0,1) circle (1pt) node[anchor=east] {\scriptsize $2$};
\filldraw[black] (1,2) circle (1pt) node[anchor=south] {\scriptsize $3$};
\filldraw[black] (2,1) circle (1pt) node[anchor=west] {\scriptsize $4$};
\filldraw[black] (2,0) circle (1pt) node[anchor=west] {\scriptsize $5$};
\end{tikzpicture}
\captionof{figure}{$\psi'$}
\label{fignoncritical1b}
\end{minipage}

\bigskip

\begin{minipage}{.5\textwidth}
\centering
\begin{tikzpicture}[scale=1.2]
\draw[thick] (0,0) -- (0,1);
\draw[thick] (2,0) -- (2,1);
\draw[thick] (0,0) -- (2,0);
\draw[mygray,thick] (0,1) -- (1,2);
\draw[thick] (1,2) -- (2,1);

\draw[mygray,thick] (0,0) -- (1,2);
\draw[mygray,thick] (0,0) -- (2,1);
\draw[thick] (0,1) -- (2,1);
\draw[mygray,thick] (0,1) -- (2,0);
\draw[mygray,thick] (1,2) -- (2,0);
\draw[thick] (0,0) -- (1,-1);
\draw[thick] (1,2) -- (1,-1);
\draw[mygray,thick] (0,1) -- (1,-1);
\draw[mygray,thick] (2,1) -- (1,-1);
\draw[mygray,thick] (2,0) -- (1,-1);

\filldraw[black] (0,0) circle (1pt) node[anchor=east] {\scriptsize $1$};
\filldraw[black] (0,1) circle (1pt) node[anchor=east] {\scriptsize $2$};
\filldraw[black] (1,2) circle (1pt) node[anchor=south] {\scriptsize $3$};
\filldraw[black] (2,1) circle (1pt) node[anchor=west] {\scriptsize $4$};
\filldraw[black] (2,0) circle (1pt) node[anchor=west] {\scriptsize $5$};
\filldraw[black] (1,-1) circle (1pt) node[anchor=north] {\scriptsize $6$};

\end{tikzpicture}
\captionof{figure}{$\varphi'$}
\label{fignoncritical2}

\end{minipage}
\begin{minipage}{.5\textwidth}
\centering
\begin{tikzpicture}[scale=1.2]
\draw[thick] (0,0) -- (0,1);
\draw[thick] (2,0) -- (2,1);
\draw[thick] (0,0) -- (2,0);
\draw[mygray,thick] (0,1) -- (1,2);
\draw[mygray,thick] (1,2) -- (2,1);

\draw[thick] (0,0) -- (1,2);
\draw[mygray,thick] (0,0) -- (2,1);
\draw[thick] (0,1) -- (2,1);
\draw[mygray,thick] (0,1) -- (2,0);
\draw[mygray,thick] (1,2) -- (2,0);
\draw[mygray,thick] (0,0) -- (1,-1);
\draw[thick] (1,2) -- (1,-1);
\draw[mygray,thick] (0,1) -- (1,-1);
\draw[thick] (2,1) -- (1,-1);
\draw[mygray,thick] (2,0) -- (1,-1);

\filldraw[black] (0,0) circle (1pt) node[anchor=east] {\scriptsize $1$};
\filldraw[black] (0,1) circle (1pt) node[anchor=east] {\scriptsize $2$};
\filldraw[black] (1,2) circle (1pt) node[anchor=south] {\scriptsize $3$};
\filldraw[black] (2,1) circle (1pt) node[anchor=west] {\scriptsize $4$};
\filldraw[black] (2,0) circle (1pt) node[anchor=west] {\scriptsize $5$};
\filldraw[black] (1,-1) circle (1pt) node[anchor=north] {\scriptsize $6$};

\end{tikzpicture}
\captionof{figure}{$\psi'$}
\label{fignoncritical2b}

\end{minipage}

\end{ex}

\bigskip

We have seen  that the finite changes of some reconstructible colorings remain reconstructible (see Proposition \ref{finitechangesEi}). The  following generalization of  Example \ref{nonrecovered}  shows an analogous fact for some non-reconstructible colorings.

\begin{prop}\label{Two_parts_unrecoverable} Let $\varphi$ be the coloring associated to a partition of $\N$ into two parts. Then,

\begin{enumerate}
\item[(i)] $\varphi_a\in\neg \UR$, for every finite set $a\subset[\N]^2$.

\item[(ii)] For every nonempty set $I\subseteq\N$, there is $\varphi_I\in\neg \UR\setminus\{\varphi,1-\varphi\}$. 
\end{enumerate}
\end{prop}

\begin{proof} 
Let $\N=A\cup B$ be a partition of $\N$, and $\varphi:[\N]^2\longrightarrow 2$ be the coloring associated to the partition.

\begin{itemize}

\item[(i)] Consider $a\subset[\N]^2$ a nonempty finite set. Let $m=\max(\bigcup a)$, $p\in A$ with $p>m$ and $q\in B$ with $q>m$. Notice that $\{p,z\},\{q,z\}\notin a$ for every $z\in\N\setminus\{p,q\}$. Thus, $\varphi_a(\{p,z\})=\varphi(\{p,z\})$ and $\varphi_a(\{q,z\})=\varphi(\{q,z\})$ for every $z\in\N\setminus\{p,q\}$. Then, $\{p,q\}$ is critical for $\varphi_a$, thus by Proposition \ref{criterio-no-UR}, $\varphi_a\in\neg\UR$.

\item[(ii)] Let $A=\{a_i:i\in\N\}$ and $B=\{b_i:i\in\N\}$ be enumerations of $A$ and $B$, and consider $\emptyset\neq I\subseteq\N$. Define $\varphi_I:[\N]^2\longrightarrow 2$ by $\varphi_I(\{x,y\})=\varphi(\{x,y\})$ if $\{x,y\}\neq\{a_n,b_n\}$ for any $n\in I$; and $\varphi_I(\{a_n,b_n\})=1$ for every $n\in I$. Then, for  $n\in I$,  $\{a_n,b_{n}\}$ is critical for $\varphi_I$ and we are done by Proposition \ref{criterio-no-UR}.
\end{itemize}
\end{proof}

We have seen in Proposition \ref{extensionUR} that any  coloring can be extended to a coloring belonging to $\UR$. Our next result shows that it can also be extended to a coloring in $\neg\UR$. 

\begin{prop}\label{extentioninnotR}
Let $\varphi$ be a coloring on $X$ and $a\not\in X$. There is a coloring $\psi$ on $X\cup\{a\}$ such that $\varphi\subset\psi$ and $\psi\in\neg \UR$. 

\end{prop}

\begin{proof}
Fix $x_0\in X$ and $a\notin X$. Let $\psi$ be a coloring on $X$ defined by $\psi(\{a,x_0\})=1$, $\psi(\{a,x\})=1-\varphi(\{x_0,x\})=0$ for $x\in X\setminus\{x_0\}$, and  $\psi|_X=\varphi$. Then, $\{a, x_0\}$ is critical for $\psi$. Hence, $\psi\in\neg\UR$ by Proposition \ref{criterio-no-UR}. 
\end{proof}

\section{Colorings associated to partitions of $\N$ into more than two parts}

In this section we will show that, in contrast with Proposition \ref{Two_parts_unrecoverable}, the coloring associated to any partition of $\N$ into at least three parts belongs to $\UR$ (Theorem \ref{colorpartition}). Furthermore, we will provide conditions on the partition so that the finite changes of the corresponding coloring are also in $\UR$ (Proposition \ref{infinitepartition} and Proposition \ref{3infiniteparts}). 

\begin{teo}
\label{colorpartition} Let $(A_i)_{i\in I}$ be a partition of $\N$ with $|I|\geq 3$. Then, the coloring associated to the partition belongs to $\UR$. 
\end{teo}

\begin{proof} Let  $(A_i)_{i\in I}$ be a partition as in the hypothesis. We will use Proposition \ref{4suffices} to show that $\varphi\in \UR$. Let $F\subseteq \N$ be  a set with 4 elements. There are two cases to be considered. If $F$ is homogeneous, then $\varphi|_F\in \UR$. Otherwise, there are  $i,j,k\in I$ such that $F\subseteq A_i\cup A_j\cup A_k$, $|F\cap A_i|\leq 3$, $|F\cap A_j|\leq 2$ and $|F\cap A_k|\leq 1$. Thus, there is $Y$ such that $F\subseteq Y\subset A_i\cup A_j\cup A_k$ such that $\varphi|_Y$ is (isomorphic to) the coloring in Example \ref{particion} and hence $\varphi|_Y\in \UR$. 
\end{proof}

The following example shows that  Proposition \ref{4suffices} cannot be  strengthened in the following sense. It can happen that a coloring $\varphi$ is reconstructible but there is $F\subseteq X$ with $|F|\leq 4$ and  $\varphi|_F\not\in \UR$.

\begin{ex}
\label{4sufficesB}  Let $\N=A\cup B\cup C$ be a partition of $\N$ into infinite sets, and $\varphi$ be the coloring associated to the partition. By Theorem \ref{colorpartition}, $\varphi\in \UR$. However,  if $x,y\in A$, $z,w\in B$ and $F=\{x,y,z,w\}$, then $\varphi|_F\not\in\UR$ as $\{x,z\}$ is critical for $\varphi|_F$. 
\end{ex}

In the following we deal with the finite changes of the coloring associated  to partition of $\N$. We  show  that whether a finite change of such coloring is in  $\UR$  depends on the type of partitions.

\begin{prop}\label{infinitepartition} 
Let $(A_k)_{k}$ be an infinite partition of $\N$. 
Then, every finite change of the coloring associated to the partition belongs to $\UR$.
\end{prop}

\begin{proof}
Let $a\subset[\N]^2$ be a finite set and $\varphi$ be the coloring on $\N$ associated to the partition $(A_k)_{k}$. We claim that $\varphi_a$ has the property $E_0$ and thus it is in $\UR$, by Proposition \ref{EiUR}. Let $F\subseteq \N$ be a  finite set. Let $k$ be  such that $(F\cup\{x,y\})\cap A_k=\emptyset$ for all $\{x,y\}\in a$. Pick $z\in A_k$. Then $\varphi_a(\{z,w\})=\varphi(\{z,w\})=0$ for all $w\in F$.
\end{proof}

\begin{prop}
\label{3infiniteparts} 
Let $(A_i)_{i<k}$ be a finite partition of $\N$, where $k>2$, and at least three $A_i$ are infinite. Then, every  finite change of the coloring associated to the partition is in  $\UR$.
\end{prop}

\begin{proof}
Let  $a\subset[\N]^2$ be a finite set and $\varphi$ be the coloring associated to the partition $(A_i)_{i<k}$. We will use Proposition \ref{4suffices} to show that $\varphi_a\in \UR$. Let $F\subseteq \N$ be  a set with 4 elements and $G=\bigcup a$.  We consider two cases: (1) There are $i, j$  such that  $A_i$ and $A_j$ are infinite and $F\cap ( A_i\cup A_j)=\emptyset$. Let  $u\in A_i$, $v\in A_j$ such that $\{u,v\}\not\in G$. Put  $Y=F\cup \{u, v\}$.   Since $\varphi_a(\{u,v\})=\varphi_a(\{u,x\})=\varphi_a(\{v,x\})$ for all $x\in F$, $\langle Y,\varphi_a|_Y\rangle$   is (isomorphic to) the coloring in Proposition \ref{propertyE} and thus $\varphi_a|_Y\in \UR$.
(2) Let $i, j, l$ be such that $A_i$, $A_j$ and $A_l$ are infinite. Let $x\in A_j$, $y\in A_j$ and $z\in A_l$ such that $\{x,y,z\}\cap (F\cup G)=\emptyset$ Let $Y=F\cup\{x,y,z\}$. Suppose that at most one of the sets $A_i$, $A_j$ and $A_l$ is disjoint from $F$. 
 By an argument analogous to that used in Proposition \ref{particion2} it follows that $\varphi_a|_Y\in\UR$.
\end{proof}

The following example shows that Proposition \ref{3infiniteparts} is optimal in the sense that we cannot ensure reconstructibility of all finite changes of the coloring associated to finite partitions of $\N$. It is interesting, since it shows that the reconstructibility is a somewhat unstable property.

\begin{ex}\label{A_0finite} There is a partition $\N=A_0\cup A_1\cup A_2$ of $\N$, with $|A_0|=2$, such that some  finite changes of the coloring associated to it are in $\UR$ and some are  in  $\neg\UR$.

\medskip

Let $A_0=\{0,1\}$, $A_1=\{2n+1:n>0\}$, $A_2=\{2n:n>0\}$, $\varphi$  be the coloring on $\N$ associated to the partition $(A_i)_{i<3}$, and $a=\big\{\{0,4\},\{1,5\},\{4,5\}\big\}$. Notice that $\{4,5\}$ is critical for $\varphi_a$, thus  $\varphi_a\in\neg\UR$ by Proposition \ref{criterio-no-UR}.

\medskip

\begin{minipage}{.5\textwidth}
\centering
\begin{tikzpicture}
     \draw[thick] (0,0) -- (0,1);
\draw[thick] (1,0) -- (1,1);
\draw[thick] (2,0) -- (2,1);
\draw[thick] (1,1) -- (2,1);
\draw[thick] (0,0) -- (1,1);
\draw[thick] (1,1) -- (1,4);
\draw[thick] (2,1) -- (2,4);

\draw[mygray, thick] (0,1) -- (1,1);
\draw[mygray, thick] (0,1) -- (1,0);
\draw[mygray, thick] (0,0) -- (1,0);
\draw[mygray, thick] (1,0) -- (2,0);
\draw[mygray, thick] (1,0) -- (2,1);
\draw[mygray, thick] (1,1) -- (2,0);
\draw[mygray, thick] (1,2) -- (2,2);
\draw[mygray, thick] (1,3) -- (2,3);
\draw[mygray, thick] (1,4) -- (2,4);
\draw[mygray, thick] (1,1) -- (2,2);
\draw[mygray, thick] (1,2) -- (2,1);
\draw[mygray, thick] (1,3) -- (2,2);
\draw[mygray, thick] (1,2) -- (2,3);
\draw[mygray, thick] (1,3) -- (2,4);
\draw[mygray, thick] (1,4) -- (2,3);

\draw[mygray, thick] (0,0) .. controls (.5,-0.5) and (1.5,-0.5) .. (2,0);
\draw[thick] (0,1) .. controls (.5,1.5) and (1.5,1.5) .. (2,1);

\filldraw[black] (0,0) circle (1pt) node[anchor=east] {\scriptsize $0$};
\filldraw[black] (0,1) circle (1pt) node[anchor=east] {\scriptsize $1$};
\filldraw[black] (1,0) circle (1pt) node[anchor=north] {\scriptsize $2$};
\filldraw[black] (1,1) circle (1pt) node (a) at (0.8,1.15) {\scriptsize $4$};
\filldraw[black] (2,0) circle (1pt) node[anchor=west] {\scriptsize $3$};
\filldraw[black] (2,1) circle (1pt) node[anchor=west] {\scriptsize $5$};
\filldraw[black] (2,2) circle (1pt) node[anchor=west] {\scriptsize $7$};
\filldraw[black] (1,2) circle (1pt) node[anchor=east] {\scriptsize $6$};
\filldraw[black] (1,3) circle (1pt) node[anchor=east] {\scriptsize $8$};
\filldraw[black] (2,3) circle (1pt) node[anchor=west] {\scriptsize $9$};
\filldraw[black] (1,4) circle (1pt) node[anchor=east] {\scriptsize $10$};
\filldraw[black] (2,4) circle (1pt) node[anchor=west] {\scriptsize $11$};

\node (b) at (1.5,4.4) {$\vdots$};
\node (b) at (0,-1) {$A_0$};
\node (b) at (1,-1) {$A_2$};
\node (b) at (2,-1) {$A_1$};

\end{tikzpicture}
\captionof{figure}{Graph of $\varphi_a$}
\end{minipage}
\begin{minipage}{.5\textwidth}
\centering
\begin{tikzpicture}
     \draw[thick] (0,0) -- (0,1);
\draw[thick] (1,0) -- (1,1);
\draw[thick] (2,0) -- (2,1);
\draw[mygray, thick] (0,0) -- (1,1);
\draw[thick] (1,1) -- (1,3);
\draw[thick] (2,1) -- (2,3);
\draw[mygray, thick] (0,1) -- (1,1);
\draw[mygray, thick] (0,1) -- (1,0);
\draw[mygray, thick] (0,0) -- (1,0);
\draw[mygray, thick] (1,0) -- (2,0);
\draw[mygray, thick] (1,0) -- (2,1);
\draw[mygray, thick] (1,1) -- (2,0);
\draw[mygray, thick] (1,2) -- (2,2);
\draw[mygray, thick] (1,3) -- (2,3);
\draw[mygray, thick] (1,1) -- (2,2);
\draw[mygray, thick] (1,2) -- (2,1);
\draw[mygray, thick] (1,3) -- (2,2);
\draw[mygray, thick] (1,2) -- (2,3);

\draw[thick] (1,1) -- (2,1);

\draw[mygray, thick] (0,0) .. controls (.5,-0.5) and (1.5,-0.5) .. (2,0);
\draw[mygray, thick] (0,1) .. controls (.5,1.5) and (1.5,1.5) .. (2,1);

\filldraw[black] (0,0) circle (1pt) node[anchor=east] {\scriptsize $0$};
\filldraw[black] (0,1) circle (1pt) node[anchor=east] {\scriptsize $1$};
\filldraw[black] (1,0) circle (1pt) node[anchor=north] {\scriptsize $2$};
\filldraw[black] (1,1) circle (1pt) node (a) at (0.8,1.15) {\scriptsize $4$};
\filldraw[black] (2,0) circle (1pt) node[anchor=west] {\scriptsize $3$};
\filldraw[black] (2,1) circle (1pt) node[anchor=west] {\scriptsize $5$};
\filldraw[black] (2,2) circle (1pt) node[anchor=west] {\scriptsize $7$};
\filldraw[black] (1,2) circle (1pt) node[anchor=east] {\scriptsize $6$};
\filldraw[black] (1,3) circle (1pt) node[anchor=east] {\scriptsize $8$};
\filldraw[black] (2,3) circle (1pt) node[anchor=west] {\scriptsize $9$};

\node (b) at (1.5,3.3) {$\vdots$};
\end{tikzpicture}
\captionof{figure}{Graph of $\varphi_b$} 
\end{minipage}

\bigskip

On the other hand, notice that $\varphi_\emptyset=\varphi\in \UR$ by Theorem \ref{colorpartition}. A non-trivial finite change of $\varphi$ which belongs to $\UR$ is $\varphi_b$ for $b=\big\{\{4,5\}\big\}$. To see this, we argue as in the proof of Proposition \ref{3infiniteparts}. 
Let $F=\{0,1,2,3,4,5\}$. It is  easy to verify that $\varphi_b|_F\in \UR$. 

\vspace{.5cm}

\end{ex}

\section{Maximal homogeneous sets}\label{maximalsection}

In this section we explore reconstructibility of a coloring looking at the maximal homogeneous sets. Let us start by observing the obvious: for any coloring of $\N$ there are maximal homogeneous sets (by Zorn's Lemma).   For any cardinal $1\leq \kappa\leq \aleph_0$ or $\kappa=2^{\aleph_0}$ there is a coloring on $\N$ with exactly $\kappa$ maximal homogeneous sets. In fact,  a constant coloring has $\N$ as the unique maximal homogeneous set. Let $\varphi$ be  the coloring associated to a partition of $\N$ into 2 infinite pieces. Then $\varphi$ has 2 maximal homogeneous sets. For $3\leq \kappa<\aleph_0$, we left to the reader to check that  a simple finite change of $\varphi$ produces a coloring with exactly $\kappa$  maximal homogeneous sets. The coloring associated to a partition of $\N$ into 3 infinite pieces has countable many maximal homogeneous sets. Finally, the coloring associated to a partition of $\N$ into infinitely many infinite pieces has $2^{\aleph_0}$ maximal homogeneous sets.  We do not know in general how this relates to the reconstructibility of the colorings. However, we present some results when $\kappa\leq 2$.

\begin{lema}
\label{externos}
Let $\varphi$ be a coloring on $\N$. We have:
\begin{itemize}
\item[(i)] Any homogeneous set is contained in a maximal homogeneous set.

\item[(ii)] Let  $A=\{x\in \N: \;x\in H \; \text{for some maximal $H\in\h(\varphi)$}\}$. Then $|\N\setminus A|\leq 2$. 

\item[(iii)] If $H$ is a maximal homogeneous set of color $i$ and $x\nin A$, then $\{y\in H:\; \varphi(\{x,y\})=i\}$ has at most one element.

\end{itemize}
\end{lema}

\proof (i) It is a well known result that  easily  follows from Zorn's lemma. 

(ii) Towards a contradiction, suppose $v,w,z\nin A$ with $\varphi(\{v,z\})=0$ and $\varphi(\{v,w\})=\varphi(\{w,z\})=1$.
 Let $H$ be an homogeneous set.  One has to consider whether $H$ is of  color 0 or 1. Both cases are treated analogously.

 (a) Suppose $H$ is of color 0.  Let  $x_1\in H$.  Since $v\nin A$, $\{x_1, v,z\}$ is not homogeneous,  we assume w.l.o.g. that  $\varphi(\{v,x_1\})=1$. As 
$\{x_1, v,w\}$ is not homogeneous, $\varphi(\{w,x_1\})=0$. Let $x_2\in H$ different than $x_1$. Since
$\{x_1, x_2,w\}$ is not homogeneous,   $\varphi(\{w,x_2\})=1$. Analogously, we conclude that  $\varphi(\{v,x_2\})=\varphi(\{x_2,z\})=0$. Therefore $\{v,z,x_2\}$ is a 0-homogeneous set, which contradicts that $z\nin A$. 

(b) Suppose $H$ is of color 1.  Let  $x_1\in H$.  Since $v\nin A$, $\{x_1, v,z\}$ is not homogeneous and thus  $\varphi(\{z,x_1\})=1$. As 
$\{x_1, v,w\}$ is not homogeneous,  $\varphi(\{w,x_1\})=0$.   Let $x_2\in H$ be different than $x_1$. Since
$\{x_1, x_2,z\}$ is not homogeneous,   $\varphi(\{z,x_2\})=0$. 
Analogously, we conclude that  $\varphi(\{v,x_2\})=1$ and $\varphi(\{w,x_1\})=\varphi(\{w,x_2\})=0$. 
Let $x_3\in H\setminus\{x_1,x_2\}$. Then $\varphi(\{z,x_3\})=0$  and $\varphi(\{v,x_3\})=1$. 
Therefore $\{v,x_2,x_3\}$ is a 1-homogeneous set, which contradicts that $v\nin A$. 

(iii) Suppose there are $y,z\in H$ such that $\varphi(\{x,y\})=\varphi(\{x,z\})=i$.   We have that $\varphi(\{y,z\})=i$ since $H$ is of color $i$. Then  $\{x,y,z\}$ is  homogeneous, which contradicts that  $x\nin A$.
\endproof

We extend the definition of the finite changes $\varphi_a$ of a coloring as follows. For each coloring $\varphi$ on $\N$ and  $A\subseteq [\N]^2$, let $\varphi_A$ be given by $\varphi_A(s)=\varphi(s)$ if $s\nin A$ and $\varphi_A(s)=1-\varphi(s)$ if $s\in A$.  The next proposition characterizes the  colorings with exactly one maximal homogeneous set.

\begin{prop}
Let $\varphi$ be a non constant coloring on $\N$. Then, $\hom(\varphi)$ has exactly one maximal element if and only if one of the following holds for a constant coloring $\psi$ on $\N$.

\begin{itemize}
\item[(i)]  $\varphi=\psi_{A_1}$ where $A_1=\{\{x_0,y\}:  \;\text{$y\in \N\setminus \{x_0\}$}\}$ for some $x_0\in \N$. 

\item[(ii)] $\varphi=\psi_{A_2}$ where $A_2=\{\{x_0,y\}:  \;\text{$y\in \N\setminus \{x_0, x_1\}$}\}$ for some $x_0, x_1\in \N$. 

\item[(iii)] $\varphi=\psi_{A_3}$ where $A_3=\{\{x_0,y\}:  \;\text{$y\in \N\setminus \{x_0,x_1,x_2\}$}\}\cup\{\{x_2,y\}:  \;\text{$y\in \N\setminus \{x_0,x_2\}$}\} $ for some $x_0, x_1, x_2\in \N$.

\item[(iv)] $\varphi=\psi_{A_4}$ where $A_4=\{\{x_0,y\}:  \;\text{$y\in \N\setminus \{x_0, x_2\}$}\}\cup\{\{x_2,y\}:  \;\text{$y\in \N\setminus \{x_2\}$}\} $ for some $x_0, x_2\in \N$.

\end{itemize}
In all cases, $\{x_0,x_1\}$ is a critical pair for $\varphi$ and therefore such colorings  are  non-reconstructible. 
\end{prop}

\begin{proof}  The pictures below describe each case. It is clear that each of the colorings $\psi_{A_i}$  has exactly one maximal homogeneous set. Let $M$ be the unique maximal homogeneous set of $\varphi$, say of color $i$ and fix $x_1\in M$. Let  $A=\{x\in \N: \;x\in H \; \text{for some maximal $H\in\h(\varphi)$}\}$. Since $\varphi$ is non constant,  $|\N\setminus A|\leq 2$, by Lemma \ref{externos}. We consider two cases: 
(1) $\N\setminus A=\{x_0\}$. We have two subcases depending on the color of $\{x_0,x_1\}$. If  $\varphi(\{x_0,x_1\})=1-i$. Then $\varphi(\{x_0,x\})=1-i$ for all $x\in M$, by the uniqueness  of $M$. Then $\varphi=\psi_{A_1}$. Analogously, if 
$\varphi(\{x_0,x_1\})=i$, then $\varphi=\psi_{A_2}$. (2) $\N\setminus A=\{x_0,x_2\}$. As in case (1) we have that 
if $\varphi(\{x_0,x_1\})=i$, then $\varphi=\psi_{A_3}$. And,  if $\varphi(\{x_0,x_1\})=1-i$, then $\varphi=\psi_{A_4}$.

\bigskip

\begin{minipage}{.25\textwidth}
\centering
\begin{tikzpicture}
\draw[thick] (1,0) -- (1,1);
\draw[thick] (0,0) -- (1,1);
\draw[thick] (1,1) -- (1,4);

\draw[mygray, thick] (0,0) -- (1,1);
\draw[mygray, thick] (0,0) -- (1,2);
\draw[mygray, thick] (0,0) -- (1,0);
\draw[mygray, thick] (0,0) -- (1,3);
\draw[mygray, thick] (0,0) -- (1,4);

\filldraw[black] (0,0) circle (1pt) node[anchor=east] {\scriptsize $x_0$};
\filldraw[black] (1,0) circle (1pt) node[anchor=north] {\scriptsize $x_1$};
\filldraw[black] (1,1) circle (1pt) node (a) at (0.8,1.15) {\scriptsize $$};
\filldraw[black] (1,2) circle (1pt) node[anchor=east] {\scriptsize $$};
\filldraw[black] (1,3) circle (1pt) node[anchor=east] {\scriptsize $$};
\filldraw[black] (1,4) circle (1pt) node[anchor=east] {\scriptsize $$};

\node (b) at (1,4.4) {$\vdots$};
\end{tikzpicture}
\captionof{figure}{$\scriptstyle\psi_{A_1}$}

\end{minipage}
\begin{minipage}{.25\textwidth}
\centering
\begin{tikzpicture}
\draw[thick] (1,0) -- (1,1);
\draw[thick] (0,0) -- (1,1);
\draw[thick] (1,1) -- (1,4);

\draw[mygray, thick] (0,0) -- (1,1);
\draw[mygray, thick] (0,0) -- (1,2);
\draw[thick] (0,0) -- (1,0);
\draw[mygray, thick] (0,0) -- (1,3);
\draw[mygray, thick] (0,0) -- (1,4);

\filldraw[black] (0,0) circle (1pt) node[anchor=east] {\scriptsize $x_0$};
\filldraw[black] (1,0) circle (1pt) node[anchor=west] {\scriptsize $x_1$};
\filldraw[black] (1,1) circle (1pt) node (a) at (0.8,1.15) {\scriptsize $$};
\filldraw[black] (1,2) circle (1pt) node[anchor=east] {\scriptsize $$};
\filldraw[black] (1,3) circle (1pt) node[anchor=east] {\scriptsize $$};
\filldraw[black] (1,4) circle (1pt) node[anchor=east] {\scriptsize $$};
\node (b) at (1,4.4) {$\vdots$};
\end{tikzpicture}
\captionof{figure}{$\scriptstyle\psi_{A_2}$} 

\end{minipage}
\begin{minipage}{.25\textwidth}
\centering
\begin{tikzpicture}
\draw[thick] (1,0) -- (1,1);
\draw[thick] (0,0) -- (1,1);
\draw[thick] (1,1) -- (1,4);
\draw[thick] (0,0) -- (0,1);

\draw[mygray, thick] (0,0) -- (1,1);
\draw[mygray, thick] (0,0) -- (1,2);
\draw[mygray, thick] (0,0) -- (1,3);
\draw[mygray, thick] (0,0) -- (1,4);
\draw[thick] (0,0) -- (1,0);
\draw[mygray, thick] (0,1) -- (1,0);
\draw[mygray, thick] (0,1) -- (1,1);
\draw[mygray, thick] (0,1) -- (1,2);
\draw[mygray, thick] (0,1) -- (1,3);
\draw[mygray, thick] (0,1) -- (1,4);

\filldraw[black] (0,0) circle (1pt) node[anchor=east] {\scriptsize $x_0$};
\filldraw[black] (0,1) circle (1pt) node[anchor=east] {\scriptsize $x_2$};
\filldraw[black] (1,0) circle (1pt) node[anchor=west] {\scriptsize $x_1$};
\filldraw[black] (1,1) circle (1pt) node (a) at (0.8,1.15) {\scriptsize $$};
\filldraw[black] (1,2) circle (1pt) node[anchor=east] {\scriptsize $$};
\filldraw[black] (1,3) circle (1pt) node[anchor=east] {\scriptsize $$};
\filldraw[black] (1,4) circle (1pt) node[anchor=east] {\scriptsize $$};

\node (b) at (1,4.4) {$\vdots$};
\end{tikzpicture}
\captionof{figure}{$\scriptstyle\psi_{A_3}$} 

\end{minipage}
\begin{minipage}{.25\textwidth}
\centering
\begin{tikzpicture}
\draw[thick] (1,0) -- (1,1);
\draw[thick] (0,0) -- (1,1);
\draw[thick] (1,1) -- (1,4);
\draw[thick] (0,0) -- (0,1);

\draw[mygray, thick] (0,0) -- (1,1);
\draw[mygray, thick] (0,0) -- (1,2);
\draw[mygray, thick] (0,0) -- (1,3);
\draw[mygray, thick] (0,0) -- (1,4);
\draw[mygray,thick] (0,0) -- (1,0);
\draw[mygray, thick] (0,1) -- (1,0);
\draw[mygray, thick] (0,1) -- (1,1);
\draw[mygray, thick] (0,1) -- (1,2);
\draw[mygray, thick] (0,1) -- (1,3);
\draw[mygray, thick] (0,1) -- (1,4);

\filldraw[black] (0,0) circle (1pt) node[anchor=east] {\scriptsize $x_0$};
\filldraw[black] (0,1) circle (1pt) node[anchor=east] {\scriptsize $x_2$};
\filldraw[black] (1,0) circle (1pt) node[anchor=west] {\scriptsize $x_1$};
\filldraw[black] (1,1) circle (1pt) node (a) at (0.8,1.15) {\scriptsize $$};
\filldraw[black] (1,2) circle (1pt) node[anchor=east] {\scriptsize $$};
\filldraw[black] (1,3) circle (1pt) node[anchor=east] {\scriptsize $$};
\filldraw[black] (1,4) circle (1pt) node[anchor=east] {\scriptsize $$};

\node (b) at (1,4.4) {$\vdots$};
\end{tikzpicture}
\captionof{figure}{$\scriptstyle\psi_{A_4}$} 

\end{minipage}

\end{proof}

\bigskip

Now we analyze colorings with exactly two maximal homogeneous sets. The prototype is the coloring associated to a partition of $\N$ into two parts. We present the analysis according to the cardinality of $\N\setminus (H_1\cup H_2)$ where $H_1$ and $H_2$ are the maximal homogeneous sets. 

\begin{lema}
\label{samecolor}
Let $\varphi$ be a coloring on $\N$ such that $\h(\varphi)$ has exactly two maximal elements $H_1$ and $H_2$. Then $\varphi([H_1]^2)=\varphi([H_2]^2)$.
 \end{lema}

\begin{proof} Assume towards a contradiction that $\varphi([H_1]^2)=\{1\}$ and $\varphi([H_2]^2)=\{0\}$. In particular, $|H_1\cap H_2|\leq 1$. By Ramsey's Theorem, $\h(\varphi)$ contains an infinite set, thus we can assume that $H_1$ is infinite. Then, $|H_1\setminus H_2|=\aleph_0$ and $|H_2\setminus H_1|\geq 2$. For every $x\nin H_2$, let 
$$
L_x=\{y\in H_2: \varphi(\{x,y\})=0\}.
$$
We claim that  $|L_x|\leq 1$. Otherwise, $|\{x\}\cup L_x|\geq 3$ and thus it is a 0-homogeneous set. Therefore,   $\{x\}\cup L_x\subseteq H_2$, a contradiction. Analogously, letting  $M_y=\{x\in H_1: \varphi(\{y,x\})=1\}$ for any $y\nin H_1$, we have that  $|M_y|\leq 1$.

Let us fix $y\in H_2\setminus H_1$. Since $|M_y|\leq 1$, there are $p,q\in H_1\setminus H_2$  such that  $\varphi(\{y,p\})=\varphi(\{y,q\})=0$. Thus,  $L_p=L_q=\{y\}$. Since $|H_1\cap H_2|\leq 1$ and $|H_2|\geq 3$, fix $z\in H_2\setminus (H_1\cup\{y\})$. Thus, $z\notin L_p\cup L_q$. That is, $\varphi(\{z,p\})=\varphi(\{z,q\})=1$, and therefore $\{p,q,z\}\in\h(\varphi)$. By Proposition \ref{externos}(i), $\{p,q,z\}$ is contained in either $H_1$ or $H_2$, which is impossible.
\end{proof}

\begin{prop}
Let $\varphi$ be a coloring on $\N$ with exactly two maximal homogeneous sets $H_1$ and $H_2$ such that $\N=H_1\cup H_2$. Then, $\varphi$ is reconstructible if and only if  $H_1\cap H_2\neq \emptyset$.
\end{prop}

\proof  Clearly, $\h(\varphi)$ has exactly two maximal elements: $H_1$ and $H_2$.  By Proposition \ref{samecolor}, we assume that $H_1$ and $H_2$ are homogeneous of color 1. Suppose $H_1\cap H_2\neq\emptyset$. Let $x_0\in H_1\cap H_2$ and $y\in H_1\setminus H_2$, $z\in H_2\setminus H_1$. 
Notice that $\{x_0,y,z\}$ is not homogeneous, otherwise $\{x,y,z\}\subseteq H_1$ or $\{x,y,z\}\subseteq H_2$, which is impossible.  Thus $\varphi(\{y,z\})=0$. Then, this coloring looks similar to the one  depicted in Figure \ref{fig2max}. The points below $x_0$ are the elements in $H_1\cap H_2$. 

\bigskip

\begin{center}
\begin{tikzpicture}[scale=1.2]

\node (b) at (0,2.5) {$\vdots$};
\node (b) at (2,2.5) {$\vdots$};

\draw[thick] (0,0) -- (0,1);
\draw[thick] (0,1) -- (0,2);

\draw[thick] (2,0) -- (2,1);
\draw[mygray,thick] (0,0) -- (2,0);
\draw[mygray,thick] (0,0) -- (2,2);
\draw[mygray,thick] (0,0) -- (2,1);
\draw[thick] (2,1) -- (2,2);
\draw[mygray,thick] (0,2) -- (2,2);
\draw[mygray,thick] (0,1) -- (2,2);

\draw[thick] (0,2) -- (1,-1);
\draw[mygray,thick] (0,2) -- (2,1);
\draw[mygray,thick] (0,2) -- (2,0);
\draw[mygray,thick] (0,0) -- (2,1);
\draw[mygray,thick] (0,1) -- (2,1);
\draw[mygray,thick] (0,1) -- (2,0);
\draw[thick] (0,0) -- (1,-1);
\draw[thick] (0,1) -- (1,-1);
\draw[thick] (2,1) -- (1,-1);
\draw[thick] (2,0) -- (1,-1);
\draw[thick] (1,-1) -- (2,2);

\draw[thick] (1,-1) -- (1,-2);
\draw[thick] (1,-2) -- (1,-3);
\node (b) at (1,-3.5) {$\vdots$};

\filldraw[black] (0,0) circle (1pt) node[anchor=east] {\scriptsize $y$};
\filldraw[black] (0,1) circle (1pt) node[anchor=east] {\scriptsize $$};
\filldraw[black] (0,2) circle (1pt) node[anchor=east] {\scriptsize $$};

\filldraw[black] (2,1) circle (1pt) node[anchor=west] {\scriptsize $$};
\filldraw[black] (2,0) circle (1pt) node[anchor=west] {\scriptsize $z$};
\filldraw[black] (2,2) circle (1pt) node[anchor=west] {\scriptsize $$};

\filldraw[black] (1,-1) circle (1pt) node[anchor=west] {\scriptsize $x_0$};
\filldraw[black] (1,-2) circle (1pt) node[anchor=west] {\scriptsize $$};
\filldraw[black] (1,-3) circle (1pt) node[anchor=west] {\scriptsize $$};

\end{tikzpicture}
\captionof{figure}{}
\label{fig2max} 

\end{center}
To see that $\varphi$ is reconstructible we use Proposition \ref{4suffices}. Let $F\subset\N$ of size 4. If $F$ is not contained in an homogeneous set, then there is a set $Y$ of size 5 such $F\subset Y$ and $\varphi|_Y$ is isomorphic to the coloring given in Example \ref{recmax}. Then, $\varphi$ is reconstructible by Proposition \ref{4suffices}. 

Conversely, suppose now that $H_1\cap H_2=\emptyset$. Let $x\in H_1$. Then $\{y\in H_2:\; \varphi\{x,y\}=1\}$ has at most one point.  If  for some $x\in H_1$, there is $y\in H_2$ such that $\varphi(\{x,y\})=1$, then $\{x,y\}$ is a critical pair. Otherwise, if $\varphi(\{x,y\})=0$ for all $x\in H_1$ and all $y\in H_2$, then  $\{x,y\}$ is critical for any $x\in H_1$ and  $y\in H_2$. In any case, $\varphi$ is non-reconstructible by Proposition \ref{criterio-no-UR}. 
\endproof

\bigskip

Now we treat the case where $|\N\setminus (H_1\cup H_2)|=1$. Before presenting a general result, we give  an example illustrating this case.

\begin{ex}\label{exAnonempty}  Define $\varphi:[\N]^2\longrightarrow 2$ by 
$$
\varphi(\{n,m\})=\begin{cases}
			0, & \text{if $n=0$ and $m>1$;}\\
			0, & \text{if $n=1$ and $m>0$ is even;}\\
			1, & \text{if $n=0$ and $m=1$;}\\
			1, & \text{if $n=1$ and $m>0$ is odd;}\\
            1, & \text{otherwise.}
		 \end{cases}
$$
See Figure \ref{fig1externo}. 
It is not difficult to see that the only maximal homogeneous sets are $H_1=\N\setminus\{0,1\}$ and $H_2=\{2n+1:n\in \N\}$.  
\begin{center}
\begin{tikzpicture}[scale=1]
\draw[thick] (0,0) -- (0,5);
\draw[thick] (2,1.5) -- (3,1.5);
\draw[thick] (0,1) -- (2,1.5);
\draw[thick] (0,3) -- (2,1.5);
\draw[thick] (0,5) -- (2,1.5);

\draw[mygray, thick] (0,0) -- (2,1.5);
\draw[mygray, thick] (0,2) -- (2,1.5);
\draw[mygray, thick] (0,4) -- (2,1.5);

\draw[mygray, thick] (0,0) .. controls (2,.5) and (2.5,1) .. (3,1.5);
\draw[mygray, thick] (0,5) .. controls (1,4.5) and (2,4) .. (3,1.5);

\filldraw[black] (0,0) circle (1pt) node[anchor=east] {\scriptsize $2$};
\filldraw[black] (0,1) circle (1pt) node[anchor=east] {\scriptsize $3$};
\filldraw[black] (0,2) circle (1pt) node[anchor=east] {\scriptsize $4$};
\filldraw[black] (0,3) circle (1pt) node[anchor=east] {\scriptsize $5$};
\filldraw[black] (0,4) circle (1pt) node[anchor=east] {\scriptsize $6$};
\filldraw[black] (0,5) circle (1pt) node[anchor=east] {\scriptsize $7$};
\filldraw[black] (2,1.5) circle (1pt) node[anchor=north] {\scriptsize $1$};
\filldraw[black] (3,1.5) circle (1pt) node[anchor=north] {\scriptsize $0$};

\node (b) at (0,5.3) {$\vdots$};
\end{tikzpicture}
\captionof{figure}{Partial drawing of  $\varphi$}    
\label{fig1externo}
\end{center}
\end{ex}

\bigskip

\begin{prop}
Let $\varphi$ be a coloring on $\N$ with exactly two maximal homogeneous sets $H_1$ and $H_2$ and such that $\N\setminus(H_1\cup H_2)=\{z\}$ for some $z$. Then, $\varphi$ has a critical pair and therefore is non-reconstructible.
\end{prop}

\proof By Proposition \ref{samecolor}, we assume that $H_1$ and $H_2$ are homogeneous of color 1.We have to consider several cases. 

(i) Suppose  there is  $x_0\in H_1$ and  $y_0\in H_2$ such $\varphi(\{z,x_0\})=\varphi(\{z,y_0\})=1$. We claim that $\{x_0, z\}$ is a critical pair for $\varphi$. 
By a simple argument (as in the proof of Proposition \ref{externos}) we have that such $x_0$ and $y_0$ are unique. Thus
\begin{equation}\label{2max}
\varphi(\{z,x\})=\varphi(\{z,y\})=0\;\; \text{for all $x\in H_1\setminus\{x_0\}$ and $y\in H_2\setminus\{y_0\}$}.
\end{equation}
 We claim that $(H_1\cup H_2)\setminus\{x_0,y_0\}$ is homogeneous. In fact, if $x\in H_1\setminus\{x_0\}$ and $y\in H_2\setminus\{y_0\}$, then $\{z,x,y\}$ is not homogeneous. By \eqref{2max},  $\varphi(\{x,y\})=1$ and from this the claim follows.  By the maximality of $H_1$ and $H_2$, we can assume w.l.o.g. that $H_2\setminus\{y_0\}\subseteq H_1$. 
 
Notice that $\N=(H_1\setminus\{x_0\})\cup\{x_0,y_0, z\}$. Let $p\nin \{x_0,z\}$.   If $p\neq y_0$,  by  \eqref{2max}, $\varphi(\{z, p\})=0$ and $\varphi(\{x_0,p\})=1$ as $p\in H_1$. On the other hand, $\varphi(\{y_0,z\})=\varphi(\{z,x_0)=1$ and $\{x_0, y_0, z\}$ is not homogeneous. Thus  $\varphi(\{x_0, y_0\})=0$. This shows that $\{x_0, z\}$ is a critical pair.

\medskip

(ii)  Suppose  there is  $x_0\in H_1$ such $\varphi(\{z,x_0\})=1$ and for all $y\in H_2$, $\varphi(\{z,y\})=0$.  By a similar argument as before one can show that 
$H_1\setminus\{x_0\}\subseteq H_2$. Let $y_0\in H_2\setminus H_1$. Then we consider two subcases. If $\varphi(\{x_0, y_0\})=1$, then  $\{x_0,y_0\}$ is a critical pair. And, if $\varphi(\{x_0, y_0\})=0$, then  $\{z,y_0\}$ is a critical pair. 

By the symmetry of the problem, we are left with the case where $\varphi(\{z,x\})=\varphi(\{z,y\})=0$ for all $x\in H_1$ and all $y\in H_2$. This implies that $H_1\cup H_2$ is homogeneous, which is imposible by the maximality of $H_1$ and $H_2$. In fact, given $x\in H_1$ and $y\in H_2$ different, we have that $\{x,y,z\}$ is not homogeneous, thus $\varphi(\{x,y\})=1$. 

\endproof

We do not have a general result about colorings such that $\N\setminus(H_1\cup H_2)$ has two elements. We just present an example which seems interesting as it is non-reconstructible but does not have a critical pair. 

\begin{ex}The coloring $\varphi$ (see Figure \ref{6.7}) has two maximal homogeneous sets,   is non-reconstructible and has no critical pair. 
 
\begin{minipage}{.5\textwidth}\centering
\begin{tikzpicture}[scale=1]

\node (b) at (0,2.5) {$\vdots$};

\draw[thick] (0,0) -- (0,1);
\draw[thick] (0,1) -- (0,2);

\draw[thick] (0,0) -- (-1,-1);
\draw[thick] (-1,-1) -- (-1,-2);
\draw[mygray,thick] (-1,-2) -- (1,-1);
\draw[mygray,thick] (-1,-1) -- (1,-1);
\draw[thick] (-1,-2) -- (1,-2);
\draw[mygray,thick] (-1,-2) -- (0,0);
\draw[mygray,thick] (0,0) -- (1,-2);
\draw[mygray,thick] (-1,-1) -- (1,-2);

\draw[mygray,thick] (-1,-2) -- (0,1);
\draw[mygray,thick] (-1,-2) -- (0,2);
\draw[mygray,thick] (1,-2) -- (0,1);
\draw[mygray,thick] (1,-2) -- (0,2);

\draw[thick] (-1,-1)..  controls (-0.7,0.5).. (0,1);
\draw[thick] (-1,-1) ..  controls (-0.7,1.5).. (0,2);
\draw[thick] (1,-1) ..  controls (0.7,0.5).. (0,1);
\draw[thick] (1,-1)..  controls (0.7,1.5)..  (0,2);
\draw[thick] (0,0) -- (1,-1);
\draw[thick] (1,-1) -- (1,-2);

\filldraw[black] (0,0) circle (1pt) node[anchor=east] {\scriptsize $4$};
\filldraw[black] (0,1) circle (1pt) node[anchor=east] {\scriptsize $5$};
\filldraw[black] (0,2) circle (1pt) node[anchor=east] {\scriptsize $6$};

\filldraw[black] (1,-1) circle (1pt) node[anchor=west] {\scriptsize $3$};
\filldraw[black] (-1,-1) circle (1pt) node[anchor=east] {\scriptsize $2$};
\filldraw[black] (1,-2) circle (1pt) node[anchor=west] {\scriptsize $1$};
\filldraw[black] (-1,-2) circle (1pt) node[anchor=east] {\scriptsize $0$};

\end{tikzpicture}
\captionof{figure}{$\varphi$}   
\label{6.7}

\end{minipage}
\begin{minipage}{.5\textwidth}
\centering
\begin{tikzpicture}[scale=1]

\node (b) at (0,2.5) {$\vdots$};

\draw[thick] (0,0) -- (0,1);
\draw[thick] (0,1) -- (0,2);

\draw[thick] (0,0) -- (-1,-1);
\draw[mygray,thick] (-1,-1) -- (-1,-2);
\draw[thick] (-1,-2) -- (1,-1);
\draw[mygray,thick] (-1,-1) -- (1,-1);
\draw[thick] (-1,-2) -- (1,-2);
\draw[mygray,thick] (-1,-2) -- (0,0);
\draw[mygray,thick] (0,0) -- (1,-2);
\draw[thick] (-1,-1) -- (1,-2);

\draw[mygray,thick] (-1,-2) -- (0,1);
\draw[mygray,thick] (-1,-2) -- (0,2);
\draw[mygray,thick] (1,-2) -- (0,1);
\draw[mygray,thick] (1,-2) -- (0,2);

\draw[thick] (-1,-1)..  controls (-0.7,0.5).. (0,1);
\draw[thick] (-1,-1) ..  controls (-0.7,1.5).. (0,2);
\draw[thick] (1,-1) ..  controls (0.7,0.5).. (0,1);
\draw[thick] (1,-1)..  controls (0.7,1.5)..  (0,2);

\draw[thick] (0,0) -- (1,-1);

\draw[mygray,thick] (1,-1) -- (1,-2);

\filldraw[black] (0,0) circle (1pt) node[anchor=east] {\scriptsize $4$};
\filldraw[black] (0,1) circle (1pt) node[anchor=east] {\scriptsize $5$};
\filldraw[black] (0,2) circle (1pt) node[anchor=east] {\scriptsize $6$};

\filldraw[black] (1,-1) circle (1pt) node[anchor=west] {\scriptsize $3$};
\filldraw[black] (-1,-1) circle (1pt) node[anchor=east] {\scriptsize $2$};
\filldraw[black] (1,-2) circle (1pt) node[anchor=west] {\scriptsize $1$};
\filldraw[black] (-1,-2) circle (1pt) node[anchor=east] {\scriptsize $0$};

\node (b) at (0,-2.5) {$\scriptstyle\psi$};

\end{tikzpicture}
\captionof{figure}{$\psi$}   
\label{6.7b}

\end{minipage}

Then,  $H_1=\{2,4,5,6,\cdots\}$ and $H_2=\{3,4,5,6,\cdots\}$ are the only maximal homogeneous sets and $\N\setminus (H_1\cup H_2)=\{0,1\}$.  There is no critical pair for $\varphi$ but $\psi$ (see Figure \ref{6.7b}) is a non-trivial reconstruction of $\varphi$. 
\end{ex}

We finish this section with an example of a non-reconstructible coloring with three pairwise disjoint infinite maximal homogeneous sets. Contrarily, we have proved in Theorem \ref{colorpartition} that the coloring associated to any partition of $\N$ into three parts belongs to $\UR$. Thus, knowing that   $\h(\varphi)$ has at least three infinite maximal pairwise disjoint elements does not guarantee that $\varphi\in\UR$.

\begin{ex}
\label{3max-no-UR}
Let $A=\{a_i:i\in\N\}$, $B=\{b_i:i\in\N\}$ and $C=\{c_i:i\in\N\}$ be pairwise disjoint subsets of $\N$ such that $\N=A\cup B\cup C$. Define $\varphi:\Nt\longrightarrow 2$ by $\varphi(\{a_i,a_j\})=\varphi(\{b_i,b_j\})=\varphi(\{c_i,c_j\})=1$ for every $i\neq j$; $\varphi(\{a_0,b_0\})=\varphi(\{a_0,c_{2i}\})=\varphi(\{b_0,c_{2i+1}\})=1$ for every $i\in\N$; and $\varphi(\{n,m\})=0$ for any other $\{n,m\}\in\Nt$. See Figure \ref{6.8}.

\begin{center}
    \begin{tikzpicture}[thick, scale=0.9]
\draw[mygray, thick] (0,0) -- (2,1);
\draw[mygray, thick] (0,1) -- (2,2);
\draw[mygray, thick] (0,2) -- (2,1);
\draw[mygray, thick] (0,1) -- (2,0);
\draw[mygray, thick] (1,-1) -- (2,0);
\draw[mygray, thick] (0,0) -- (1,-2);
\draw[mygray, thick] (0,2) -- (2,2);
\draw[mygray, thick] (0,1) -- (2,1);
\draw[mygray, thick] (0,3) -- (2,2);
\draw[mygray, thick] (0,2) -- (2,3);
\draw[mygray, thick] (0,3) -- (2,3);

\draw[thick] (0,0) -- (0,3);
\draw[thick] (1,-1) -- (1,-4);
\draw[thick] (2,3) -- (2,0);
\draw[thick] (0,0) -- (2,0);
\draw[thick] (0,0) -- (1,-1);
\draw[thick] (2,0) -- (1,-2);

\draw[mygray, thick] (2,0) .. controls (1.5,-2) .. (1,-3);
\draw[mygray, thick] (0,0) .. controls (0.4,-3) .. (1,-4);
\draw[thick] (0,0) .. controls (0.5,-2) .. (1,-3);
\draw[thick] (2,0) .. controls (1.6,-3) .. (1,-4);

\filldraw[black] (0,0) circle (1pt) node[anchor=east] {\scriptsize $a_0$};
\filldraw[black] (0,1) circle (1pt) node[anchor=east] {\scriptsize $a_1$};
\filldraw[black] (0,2) circle (1pt) node[anchor=east] {\scriptsize $a_2$};
\filldraw[black] (0,3) circle (1pt) node[anchor=east] {\scriptsize $a_3$};
\filldraw[black] (2,0) circle (1pt) node[anchor=west] {\scriptsize $b_0$};
\filldraw[black] (2,1) circle (1pt) node[anchor=west] {\scriptsize $b_1$};
\filldraw[black] (2,2) circle (1pt) node[anchor=west] {\scriptsize $b_2$};
\filldraw[black] (2,3) circle (1pt) node[anchor=west] {\scriptsize $b_2$};
\filldraw[black] (1,-1) circle (1pt) node[anchor=west] {\scriptsize $c_0$};
\filldraw[black] (1,-2) circle (1pt) node[anchor=west] {\scriptsize $c_1$};
\filldraw[black] (1,-3) circle (1pt) node[anchor=west] {\scriptsize $c_2$};
\filldraw[black] (1,-4) circle (1pt) node[anchor=west] {\scriptsize $c_3$};

\node (b) at (1,3.3) {$\vdots$};
\node (b) at (1,-4.3) {$\vdots$};
\end{tikzpicture}
\captionof{figure}{Partial drawing of  $\varphi$}   
\label{6.8}
\end{center}

\medskip

Notice that $A,B$ and $C$ are maximal elements of $\h(\varphi)$. Moreover, $\{a_0,b_0\}$ is critical for $\varphi$, thus  $\varphi\in\neg\UR$, by Proposition \ref{criterio-no-UR}.
\end{ex}

\section{Complexity of the  reconstruction problem.}

This section is devoted to analyzing the reconstruction problem from the descriptive set theoretic point of view.  We  show that the problem of recovering  a coloring from the collection of its homogeneous sets can be done in a Borel way. 

The space of colorings $2^{[\N]^2}$ is endowed with the product topology which has as basic open sets $\{\varphi\in 2^{[\N]^2}: \; \varphi_0\subseteq \varphi\}$ for $\varphi_0$ a coloring on a finite subset of $\N$.  Let $K(\cantor)$ be the hyperspace of compact subsets of $\cantor$ with the Vietoris topology (see, for instance, \cite[4F]{Kechris94}). A subbasis for $K(\cantor)$ consists of the sets $V^+=\{L\in K(\cantor): \; L\subset V\}$ and $V^-=\{L\in K(\cantor): \; L\cap V\neq \emptyset\}$, for $V\subseteq \cantor$ open. As $\cantor$ is zero dimensional, we can also assume that $V$ is clopen.  We recall that a subset of a topological space is $G_\delta$ (respectively, $F_\sigma$), if it is a countable intersection of open sets (respectively, a countable union of closed sets).

 Under the usual identification of a subset of $\N$ with its characteristic function,  $\hom(\varphi)$ is not topologically closed in $\cantor$. Notice, however,  that  if $A\in cl(\h(\varphi))\setminus\h(\varphi)$, then $|A|\leq 2$.  Thus $cl(\h(\varphi))= cl(\h(\psi))$ iff $\h(\varphi)=\h(\psi)$.  Since we want to analyze the reconstruction problem from a topological point of view, it is better to work  with a closed set instead of  $\h(\varphi)$. However,   it is more convenient to use  a closed set larger than the closure. Let 
 \[
 \ch(\varphi)=\h(\varphi)\cup\{A\subset \N: |A|\leq 2\}.
 \]
 Notice that $cl(\h(\varphi))\subseteq \ch(\varphi)$ and $\ch(\varphi)$ is closed. Also, $\ch(\varphi)= \ch(\psi)$ iff $\h(\varphi)=\h(\psi)$. 
 
The main result of this section is to show that there is a Borel function $g:K(\cantor)\to 2^{[\N]^2}$ such that 
\[
\ch(g(\ch(\varphi)))=\ch(\varphi)
\]
for all $\varphi\in 2^{[\N]^2}$.
So $g(\ch(\varphi))$ is  a coloring recovered  from $\ch(\varphi)$, but notice that $g(\ch(\varphi))$ might be neither $\varphi$ nor  $1-\varphi$ when $\varphi\not\in \UR$. 

\begin{prop}
\label{RecobrableBorel}
The collection of all colorings on $\N$ belonging to $\UR$  is a dense $G_\delta$ subset of $2^{[\N]^2}$.
\end{prop}

\begin{proof}
Recall that $\h(\varphi)=\h(\psi)$ if and only if  $\h(\varphi)\cap [\N]^{3}=\h(\psi)\cap [\N]^{3}$ (see Proposition \ref{triangulos}).
The following set is closed:
\[
E=\{(\varphi,\psi)\in 2^{[\N]^2}\times  2^{[\N]^2}: \; \h(\varphi)=\h(\psi)\}.
\]
In fact, $(\varphi,\psi)\not\in E$ if and only if   there is $H\in [\N]^{3}$ such that either $H\in \h(\varphi)\setminus \h(\psi)$ or $H\in \h(\psi)\setminus \h(\varphi)$. For every finite set $H$, it is straightforward to verify that $ \{\varphi\in  2^{[\N]^2}: \; H\in \h(\varphi)\}$ is clopen. Thus the complement of $E$ is open.

Now consider the relation $\varphi\approx \psi$ if either $\varphi=\psi$ or  $\psi= 1-\varphi$. Clearly $\approx$ is a closed subset of  $2^{[\N]^2}\times 2^{[\N]^2}$. 
Finally we have
\[
\varphi\not\in \UR \;\Leftrightarrow \;\; \exists \psi\in  2^{[\N]^2}\; ( (\varphi,\psi)\in E \wedge (\varphi\not\approx\psi)).
\]
Thus, the collection of colorings that are not reconstructible is the  projection of a $K_\sigma$ set (i.e. a countable union of compact sets) and thus it is also $K_\sigma$. 

Finally,  from the definition of the product topology on $2^{[\N]^2}$,  given a coloring $\varphi$, we have that the collection of all its finite changes, i.e. $\{\varphi_a:\; \mbox{$a\subseteq [\N]^2$ is finite}\}$ is a dense subset of $2^{[\N]^2}$. Thus,   by Corollary \ref{finitechanges}, $\UR$ is dense.
\end{proof}

\begin{prop}
\label{funcionhom}
The function $\ch:2^{[\N]^2}\to K(2^{\N})$ given by $\varphi\mapsto \ch(\varphi)$ is Borel.
\end{prop}

\begin{proof}
Let $V$ be a clopen subset of $2^{\N}$. We  need to  show that  $C(V^+)=\{\varphi\in 2^{\Nt}:\; \ch(\varphi)\in  V^+\}$ and $C(V^-)=\{\varphi\in 2^{\Nt}:\; \ch(\varphi)\in  V^-\}$ are  Borel.  

First, if $[\N]^{\leq 2}\not\subseteq V$, then $C(V^+)=\emptyset$ and there is nothing to show. So we assume that $[\N]^{\leq 2}\subseteq V$,
In this case,  $\ch(\varphi)\subseteq V$ is equivalent to $\h(\varphi)\subseteq V$. Then, we have that
\begin{equation}
\label{vmas}
\ch(\varphi)\subseteq V \Leftrightarrow (\forall H\in [\N]^{<\omega})(H\in \h(\varphi)\rightarrow H\in V).
\end{equation}
Notice that $\{\varphi\in 2^{\Nt}:\; H\in \h(\varphi)\}$ is clopen for every finite $H\subset \N$.
From this and \eqref{vmas}  we conclude that $C(V^+)$ is $G_\delta$.

 For $C(V^-)$ notice that if $V\cap [\N]^{\leq 2}\neq \emptyset$, then $C(V^-)=2^{\Nt}$ and there is nothing to show.  Since, $[\N]^{\leq 2}$ is a closed nowhere dense subset of $\cantor$, we can assume w.l.o.g. that $V\cap [\N]^{\leq 2}=\emptyset$. In this case, we have  $\ch(\varphi)\cap V\neq\emptyset$ if and only if $\h(\varphi)\cap V\neq\emptyset$. And, as before, this happens when   $\h(\varphi)\cap [\N]^{<\omega}\cap V\neq\emptyset$. Thus $C(V^-)$ is open.
\end{proof}
 
 \medskip

Let $HOM=\{\ch(\varphi):\;\varphi\in 2^{[\N]^2}\}$.

\bigskip

\begin{prop}
\label{homgdelta}
$HOM$ is $G_\delta$ in $K(\cantor)$ and thus it is Polish.
\end{prop}

\begin{proof}
Let $\mathcal{P}_n$ denote the collection of subsets of $n+1$ of size at least 3. Let $L\in K(\cantor)$.  We claim 
\begin{eqnarray}
\label{gdelta}
L\in HOM\;\;\Leftrightarrow \;\; [\N]^{\leq 2}\subseteq L\; \&\; (\forall n\geq 2)\,(\exists\varphi\in 2^{[n+1]^2})\, ( L\cap \mathcal{P}_n=\h(\varphi)).
\end{eqnarray}

In fact, suppose $\varphi$ is a coloring on $\N$ and $L=\h(\varphi)$. Then, $\varphi|_{n+1}$ satisfies the right hand side of (\ref{gdelta}). 

For the other direction,  let $L\in K(\cantor)$ be such that $[\N]^{\leq 2}\subseteq L$ and consider the following set
\[
T_L=\{\varphi\in 2^{[n+1]^2}: \;  L\cap \mathcal{P}_n=\h(\varphi)\; \mbox{and $n\in \N$}\}.
\]
Then, $T_L$ is a finitely branching tree. If $L$ satisfies the right hand side of (\ref{gdelta}), then  $T_L$ is infinite, thus it has a branch $\varphi$ which is clearly a coloring on $\N$. Then, $L$ and $\h(\varphi)$ contain the same finite sets of size at least 3, thus $L=\ch(\varphi)$.

To see that $HOM$ is $G_\delta$ we observe the following. The collection of all $L\in K(\cantor)$ such that $ [\N]^{\leq 2}\subseteq L$ is closed. On the other hand, given $B\subseteq A\subset\cantor$ with $A$ finite. Then,  $\{L\in K(\cantor):\; L\cap A=B\}$ is $G_\delta$. 
The reason is that the relation $``x\in L"$ is closed in $\cantor\times K(\cantor)$ and we have 
\[
L\cap A=B\;\Leftrightarrow \; \forall x\in A\; (x\in L\leftrightarrow x\in B).
\]
The last claim is the classical fact that a subset of a Polish space (i.e. a completely metrizable separable space)  is itself Polish if and only if it is $G_\delta$ (see \cite[3.11]{Kechris94}).
\end{proof}

\begin{teo}
There is $g: HOM\to 2^{[\N]^2}$ Borel such that, for all $L\in HOM$,   
\[
\ch(g(L))= L.
\]
\end{teo}

\begin{proof}
Consider the following relation: 
\[
A=\{(L, \varphi)\in HOM\times 2^{[\N]^2}:\; L=\ch(\varphi)\}.
\]
Since the function $\ch$ is Borel (see Proposition \ref{funcionhom}), $A$ is Borel. 
We claim that all vertical sections of $A$ are closed (and hence compact). In fact, let $E$ be the equivalence relation on $2^{[\N]^2}$ given by $\varphi E \psi$ if $\h(\varphi)=\h(\psi)$.  We have seen in the proof of Proposition \ref{RecobrableBorel} that $E$ is closed.  Recall that $\hom(\varphi)=\hom(\psi)$ iff $\ch(\varphi)=\ch(\psi)$. Thus,  for every $L\in  HOM$, if $L=\ch(\varphi)$, then $A_L$ is the  $E$-equivalence class of $\varphi$  which is closed, as promised. 

Since $HOM$ is Polish (by Proposition \ref{homgdelta}), we can use a classical uniformization theorem to define $g$.  We know   that any Borel relation on a Polish space with $K_\sigma$ sections has a Borel uniformization (see \cite[18.18]{Kechris94}). Thus,  there is a Borel map $g:HOM\to 2^{[\N]^2}$ such that $ \ch(g(L))= L$.\end{proof}

\noindent{\bf Acknowledgment:} We thank the referees for all their comments and suggestions which help to  improve the presentation of the results.

\end{document}